\documentclass[11pt,twoside]{article}
\usepackage{latexsym}
\usepackage{amssymb,amsbsy,amsmath,amsfonts,amssymb,amscd}
\usepackage{graphicx,color}
\usepackage{float,url}
\usepackage[colorlinks,linkcolor=red,citecolor=blue]{hyperref}
\usepackage{amsmath}
\usepackage{booktabs}
\usepackage{verbatim}
\usepackage{cancel} 
\usepackage{mathrsfs} 
\usepackage{subfigure}
\usepackage{enumerate}
\usepackage{enumitem}
\usepackage{amssymb, amsmath}

\allowdisplaybreaks[4]
\newcommand{\commentout}[1]{}

\newcommand {\vp} {\varphi}

\newcommand{\beq}{\begin{equation}}
\newcommand{\eeq}{\end{equation}}
\newcommand{\bea} {\begin{array}{rl}}
\newcommand{\eea} {\end{array}}
\newcommand{\bepa}{\left\{ \begin{array}{l}}
\newcommand{\eepa} {\end{array}\right.}
\allowdisplaybreaks[4]
\newtheorem{theorem}{Theorem}
\newtheorem{lemma}[theorem]{Lemma}
\newtheorem{definition}[theorem]{Definition}
\newtheorem{remark}[theorem]{Remark}
\newtheorem{proposition}[theorem]{Proposition}

\newtheorem{notation}[theorem]{Notation}

\newcommand{\RNum}[1]{\uppercase\expandafter{\romannumeral #1\relax}}
\newcommand{\qed}{{ \hfill
                       {\unskip\kern 6pt\penalty 500 \raise -2pt\hbox{\vrule\vbox to 6pt{\hrule width 6pt
                       \vfill\hrule}\vrule} \par}   }}
\title{On the incompressible limit of Keller-Segel system with volume-filling effects}

\author{ 
Qingyou He\thanks{Sorbonne Universit{\'e}, CNRS, Laboratory of Computational and Quantitative Biology (LCQB), 75205 Paris Cedex 06, France.
Email : qingyou.he@sorbonne-universite.fr}
\and Mingyue Zhang\footnotemark[1]\thanks{Sorbonne Universit{\'e}, CNRS, Laboratoire Jacques-Louis Lions, F-75005 Paris. Email : mingyue.zhang@sorbonne-universite.fr}
}
\date{}

\begin{document}
\maketitle
\pagestyle{plain}
\pagenumbering{arabic}
\begin{abstract}
We consider the Keller-Segel system with a volume-filling effect and study its incompressible limit. Due to the presence of logistic-type sensitivity, $K=1$ is the critical threshold. When $K>1$, as the diffusion exponent tends to infinity, by supposing the weak limit of $u^2_m$, we prove that the limiting system becomes a Hele-Shaw type free boundary problem. For $K\le 1$, we justify that  the stiff pressure effect ($\Delta P_\infty$) vanishes, resulting in the limiting system being a hyperbolic Keller-Segel system. Compared to previous studies, the new challenge arises from the stronger nonlinearity induced by the logistic chemotactic sensitivity. To address this, our first novel finding is the proof of strong convergence of the density on the support of the limiting pressure, thus confirming the validity of the \emph{complementarity relation} for all  $K>0$. Furthermore, specifically for the case  $K\le1$, by introducing the \emph{kinetic formulation}, we verify the strong limit of the density required to reach the incompressible limit.
\end{abstract}
\\
2020 {\bf MSC.} 35Q92; 35R35; 35D30; 35K65; 92B05.
\section{Introduction}
We consider the Keller-Segel system with the logistic type sensitivity producing the volume-filling effect, expressed by
\begin{equation}\label{deq}
\begin{cases}
\partial_t u_m=\Delta u_m^m-\nabla\cdot(u_m(K-u_m)\nabla c_m),\quad &x\in \Omega,\ t>0,\\
-\Delta c_m+c_m=u_m,\quad &x\in \Omega,\ t>0,\\
\nabla u_m \cdot \vec{n}=\nabla c_m\cdot \vec{n}=0,\quad &x\in\partial  \Omega,\ t>0. 
\end{cases}
\end{equation}
 The solution $u_m$ denotes the cell population, $c_m$ is the concentration of chemoattractant, $m>1$ describes the diffusion exponent (slow diffusion), $K>0$ is the general converging capacity of the cell population density, $\Omega\subset \mathbb{R}^n$ ($n\geq2$) is a bounded domain with $\mathcal{C}^1$ boundary, and $\vec{n}(x)$ is
the outward normal vector to $\Omega$ at $x\in\partial \Omega$.
Another crucial variable for the porous medium type equation is the pressure
\[P_m=\frac{m}{m-1}u_m^{m-1}\quad (\text{Pressure law}),\]
which  solves the corresponding pressure equation 
\begin{equation}\label{pe}
\partial_t P_m=(m-1)P_m[\Delta P_m+(K-u_m)(u_m-c_m)]+\nabla P_m\cdot(\nabla P_m+(2u_m-K)\nabla c_m).
\end{equation} For simplicity, we set 
\begin{equation}\label{ini}u_{m,0}(\cdot):=u_m(\cdot,0)\geq0,\ P_{m,0}(\cdot):=P_{m}(\cdot,0),\ x\in\Omega.\end{equation}

\vspace{2mm}
 The main objective of this paper is to study the incompressible limit of system~\eqref{deq}. Compared to the classical Keller-Segel system with volume filling effect, i.e., $m=1$ and $K=1$ (see, e.g.,\cite{Hillen2002,bubba2020discrete,CarrilloCalvez2006,wang2007classical}),
system~\eqref{deq}  has a stronger nonlinearity for the density. It has more challenge than in the previous relevant works (see, e.g., ~\cite{PERTHAME2014,BPPS2020,DAVID2021,HLP2023,DPSV2021,DS2021}) since the cut-off effect acts on the density instead of the pressure. Therefore, to complete the incompressible limit, strong limits of both density $u_m$ and pressure gradient $\nabla P_m$ in the diffusion exponent are required.
 
 \vspace{2mm}
 



Chemotaxis refers to the active movement of cells and organisms in response to chemical gradients. This phenomenon has been modeled in numerous contexts, including bacterial movement, immune response, and cancer progression. A prominent mathematical model for chemotaxis is the Patlak-Keller-Segel system, introduced in \cite{patlak1953random, keller1970initiation}. Due to its simplicity, this model has been extensively applied across various biological systems (see \cite{Murray2003, Maini2008, Hillen2009, Painter2018}).

\vspace{2mm}

Numerous modifications have been proposed, such as incorporating a nonlinear diffusion coefficient or nonlinear chemotactic sensitivity. Many studies focus on a nonlinear diffusion coefficient dependent on~$u$, often of the porous medium type, as seen in \cite{kowalczyk2005preventing, CarrilloCalvez2006, wang2007classical}. Nonlinear chemotactic sensitivities have also been explored, with the logistic sensitivity function being a notable example. The system \eqref{deq}, which includes both nonlinear diffusion and nonlinear chemotactic sensitivity, has been proposed in \cite{bpmz2024}.

\paragraph{Related studies.}
The porous medium type equations have a limit as $m\rightarrow\infty$, namely, a weak form of the Hele-Shaw type free boundary problem. The authors in~\cite{PERTHAME2014} first extended related studies to tumor growth models, contributing to a large body of works in this direction~\cite{DAVID2021,DPSV2021,DS2020,GKM2022,KM2023,KT2018,LX2021, PV2015}. A preliminary study of the Hele-Shaw asymptotic (limit)~(\cite{DAVID2021, PERTHAME2014,PQTV2014}) for the tumor growth model is carried out using the weak solution method. The incompressible (Hele-Shaw) limit of tumor growth incorporating convective effects is verified in~\cite{DS2021}, and the decay rates on the diffusion exponent~$m$ are argued in~\cite{DDB2022,NAF2024}. For a tumor growth model where Brinkmann's pressure law governs the motion, the authors establish the optimal uniform decay rate of density and pressure in $m$ by a viscosity solution method in~\cite{KT2018}, and prove the Hele-Shaw limit for the two-species case by a compactness technique in~\cite{DPSV2021,DS2020}. In~\cite{GKM2022,KM2023}, the Hele-Shaw limit for the non-monotonic (or even non-local) reaction term of the porous medium equation is obtained using the obstacle problem approach. Weak solutions of tissue growth models with autophagy and the existence of a free boundary (Hele-Shaw) limit have been obtained in~\cite{LX2021}. For the porous medium equation with a drift, the singular limit is studied by the viscosity solution method in~\cite{KPW2019}. The convergence of the free boundary in the incompressible limit of tumor growth with convective effects is recently achieved in \cite{ZT2014}. A rigorous derivation of a Hele-Shaw type model and its non-symmetric travelling wave solution for the tumor growth model are given in \cite{fhlz2024}. The Hele-Shaw limit for the Patlak-Keller-Segel model is first proved in \cite{CKY2018} by combining the viscosity solution and the gradient flow, then in \cite{HLP2023,HLP2022} (even with a growth term)  by the weak solution method.
\vspace{3mm}

\paragraph{Understanding the difficulties.} For the Keller-Segel system with volume filling effects~\eqref{deq}, the stronger nonlinearity of this density leads to the main difficulty in the process of the incompressible limit. To understand this difficulty, we present some  classical  approaches to studying the incompressible limit for the related problems, and then explain that these do not apply to 
\eqref{deq}. The case with linear sensitivity function \eqref{deq}, called Patlak-Keller-Segel system, has been studied in \cite{HLP2022} by the method of weak solution, and in \cite{CKY2018} by combining the gradient flow and the viscosity solution. Recently, \cite{HLP2022} has been extended to chemotaxis with growth~\cite{HLP2023}, whose growth (source) term, producing the logistic effect, is the product of the density and the function of the pressure. For these types of source term and linear sensitive function, the strong convergence of the pressure is sufficient to pass to the limit on $m$, as in \cite{HLP2023,noemi2023,HLP2022}, and the strong limit of the density is not necessary. For the reaction-diffusion equations modelling tumor growth \cite{PERTHAME2014,DAVID2021,BPPS2020}, the $BV$ regularity of the density can be obtained by using Kato's inequality, which further gives the strong convergence of  density by the compact embedding.    For the non-monotonic source terms, the authors in \cite{KM2023} use the viewpoint of the obstacle problem to verify the complementarity relation, the $BV$ regularities of density and pressure are still required. Due to the presence of the aggregation effect for the Keller-Segel system, it is difficult to obtain the $BV$ regularity of the density; see~\cite{HLP2023,HLP2022,CKY2018}. We can shed light on the $BV$ regularity of pressure on the chemotaxis as in \cite{HLP2022}.  When the diffusion is governed by Brinkman's law, the authors in \cite{DPSV2021} use the compactness criteria denoted by an integral formula to show the strong limit of the density, and also do so for the pressure by analysing the corresponding kinetic variable. For system~\eqref{deq}, even though the pressure gradient is one of the velocities, the regularity of the pressure gradient is not enough for this kind of integral-type compactness criterion, as in \cite{DPSV2021}. In \cite{IN2021}, the authors use an accretive operator method to obtain the strong convergence of density. However, for system \eqref{deq}, the non-local aggregation effect causes the loss of the uniform spectral radius of the accretive operator. Therefore, a new idea or approach is required to address this issue.  

\paragraph{Strategies.} To achieve the incompressible limit, as previously mentioned, we need to establish the strong limits of both density and pressure gradient. Initially, we utilize the methods outlined in \cite{HLP2023,noemi2023,LX2021} to demonstrate the strong limit of the pressure and its gradient, bypassing the need for the Aronson-Bénilan type estimate as used in \cite{PERTHAME2014,DAVID2021}. The innovative aspect of this work lies in justifying the strong limit of the density on the support of the limiting pressure, without relying on regularity estimates for the density. Following this, we derive the \emph{complementarity relation}
\begin{equation}\label{comrela}P_\infty[\Delta P_\infty+(K-1)(1-c_\infty)]=0,\end{equation}
which includes the pressure in a degenerate elliptic equation. The verification of the complementarity relation is the main challenge in the previous works; for example~\cite{PERTHAME2014,BPPS2020,DAVID2021,HLP2022,LX2021,DS2020,DS2021,HLP2023,noemi2023}. The reason is that the strong convergence of the pressure gradient, which is used to deal with the nonlinearity of the pressure, generally requires higher regularity estimates of the pressure. Secondly, we find that $K=1$ is a critical threshold for the system \eqref{deq}. To be more precisely, when $K=1$, we prove that the limiting pressure gradient vanishes, i.e. $\nabla P_\infty=0$. When $0<K<1$, the limiting pressure vanishes, i.e. $P_\infty=0$. Therefore, the case with $K\leq1$ shows that the stiff pressure effect ($\Delta P_\infty$), the main feature of the Hele-Shaw problem \cite{IK2003}, vanishes, and thus the limiting equation is the so-called hyperbolic Keller-Segel system~\cite{bp2009}. Thirdly, we prove the strong limit of density. To do this, we introduce the kinetic formulation, which originates from the hyperbolic conversation law~\cite{bp2002}. More precisely, for $K\leq 1$, the vanishing of the stiff pressure effect $(\Delta P_\infty=0)$ leads to the effectiveness of the so-called kinetic formulation method in \cite{bp2009}. 

\vspace{3mm}

We introduce some notations that are useful in the rest of this paper.
\begin{notation}
Given $T>0$ and $\Omega\subset \mathbb{R}^n$, we denote
\begin{itemize}
\item $Q_T=\Omega\times (0,T)$, $\Gamma_T=\partial\Omega\times (0,T)$.
\item $C(T)>0$ is a universal constant, and depends only on $T$. 
\end{itemize}
\end{notation}
\paragraph{Initial assumptions.} We assume that the initial data $u_{m,0}=u_m(x,0)$ and $P_{m,0}=P_m(x,0)$ satisfy
\begin{align}
&0\leq u_{m,0}\leq K,\quad \|u_{m,0}\|_{L^1(\Omega)}\leq \alpha|\Omega|,  \quad m>1,\ \alpha\in(0,1),\label{a1}\\
&\|P_{m,0}\|_{L^1(\Omega)}\leq C,\quad P_{m,0}:=\frac{m}{m-1}u_{m,0}^{m-1}, \quad m>1,\label{a2}\\
&\lim\limits_{m\to\infty}\|u_{m,0}-u_{\infty,0}\|_{L^1(\Omega)}=0\quad\text{ with   }u_{\infty,0}\in L^1(\Omega).\label{a3}
\end{align}

\begin{remark}
The $L^1$ condition in \eqref{a1} may seem strong. However it is obviously necessary to establish the bound by 1 in \eqref{hss1}.
\end{remark}
\begin{definition}[Weak solution] Under the initial assumptions~\eqref{a1}, the non-negative function $u_m: \mathbb{R}^n\times (0,\infty)\to[0,\infty)$ is said to be a weak solution of system~\eqref{deq} provided that for any $T>0$,   
\[0\leq u_m\leq K,\quad u_m\in \mathcal{C}(0,T;L^1(\Omega)),\quad u_m^m\in L^2(0,T;H^1(\Omega)),\quad \text{and }\ c_m\in \mathcal{C}(0,T;W^{2,p}(\Omega))\]
with $p\in[1,\infty)$, and the system holds in the sense of distribution; i.e., for any test functions $\varphi\in \mathcal{C}^1(\Omega\times [0,T))$ with $\vp(x,T)=0$ and $\phi\in\mathcal{C}^1(\Omega)$, it holds true that 
\begin{equation*}\label{ee1}\iint_{Q_T}[-u_m \partial_t\varphi+\nabla u_m^m\cdot\nabla\varphi-u_m(K-u_m)\nabla c_m\cdot\nabla \varphi]dxdt=\int_{\Omega}u_{m,0}\varphi(x,0)dx,\end{equation*}
\begin{equation*}\label{ee2}\iint_{Q_T}[\nabla c_m\cdot\nabla\phi+(c_m-u_m)\phi] dxdt=0.\end{equation*}
\end{definition}

\paragraph{Main results.}\rm
We first state that the complementary relation holds for any $K>0$.
\begin{theorem}[Complementarity relation]\label{crelation}
Assume that \eqref{a1}, $\eqref{a2}$ and \eqref{a3} hold for the initial data $(u_{m,0},P_{m,0})$. Then, for any  $T>0$ and any $K>0$, and the weak solution $(u_m,c_m)$ of the problem \eqref{deq} and \eqref{ini}, there exist $ u_\infty,\overline{u_\infty^2}\in L^\infty(Q_T)$, $P_\infty\in L^2(0,T;H^1(\Omega))$ and $c_\infty\in L^\infty(0,T;W^{2,q}(\Omega))$ with $q\in[1,\infty)$ such that, up to a subsequence (still denoted by $u_m, c_m, P_m$), as $m\to\infty$, we have
\begin{equation}\label{ss0}
u_m,\ u_m^2\rightharpoonup u_\infty,\ \overline{u_\infty^2}, \quad *-\text{weak in }L^\infty(Q_T),
\end{equation}
\begin{equation}\label{ss1}
c_m\to c_\infty,\quad \text{strongly in }L^p(0,T;W^{1,p}(\Omega))\text{ with } p\in[1,\infty),
\end{equation}
\begin{equation}\label{ss3}
u_m\sqrt{P_\infty}\to \sqrt{P_\infty}, \quad\text{strongly in }L^2(Q_T),
\end{equation}
\begin{equation}\label{ss2}
\nabla P_m\to \nabla P_\infty,\quad\text{strongly in }L^2(0,T;H^1(\Omega)).
\end{equation}
Moreover, the complementraity relation \eqref{comrela} holds in the sense of distribution and we have
\begin{equation}\label{hss1}
\begin{cases}
\partial_t u_\infty=\Delta P_\infty-\nabla\cdot\big((Ku_\infty-\overline{u_\infty^2})\nabla c_\infty\big),\quad &\text{ in }\mathcal{D}'(Q_T),\\
-\Delta c_\infty+c_\infty=u_\infty,\ 0\leq u_\infty,\overline{u_\infty^2},c_\infty\leq 1,\quad &\text{ a.e. in }Q_T, \\
P_\infty(1-u_\infty)=0,\ \nabla P_\infty(1-u_\infty)=0,\quad &\text{ a.e. in }Q_T, \\
\nabla c_\infty\cdot \vec{n}=\nabla P_\infty\cdot \vec{n}=0,\ &\text{ on }\Gamma_T.
\end{cases}
\end{equation}
\end{theorem}

Specifically, for $K\leq1$, we show that the limiting system is a hyperbolic Keller-Segel system (cf.\cite{bp2009}).

\begin{theorem}[Hyperbolic Keller-Segel]\label{hks} Under the  assumptions \eqref{a1}-\eqref{a3}, 
for $K\leq 1$, upon the extraction of a subsequence, the stiff pressure effect vanishes, that is
\begin{equation}\label{pressurev}\nabla P_\infty=0\ \text{ for }K=1,\quad P_\infty=0\ \text{ for }K<1,\end{equation}
\begin{equation}\label{ds}
u_m\to u_\infty, \quad\text{strongly in }L^p(Q_T) \text{ with }p\in[1,\infty),\quad \text{ as }m\to\infty.
\end{equation}
Moreover, the limiting function $(u_\infty,c_\infty)$ solves the following hyperbolic Keller-Segel system with conservation law
\begin{equation}\label{hss2}
\begin{cases}
\partial_tu_\infty=-\nabla\cdot(u_\infty(K-u_\infty)\nabla c_\infty),\quad &\text{ in }\mathcal{D}'(Q_T),\\
-\Delta c_\infty+c_\infty=u_\infty,\quad 0\leq u_\infty\leq K,\quad &\text{ a.e. in }Q_T,\\
\nabla c_\infty\cdot\vec{n}=0,\ &\text{ on }\Gamma_T.
\end{cases}
\end{equation}
\end{theorem}

\begin{remark}When $K\leq 1$, it should be noticed that the incompressible limit of the porous medium type Keller-Segel system~\eqref{deq} is no longer a Hele-Shaw problem. This differs from previous works (see~e.g., \cite{PERTHAME2014,DAVID2021,BPPS2020,KM2023,HLP2022,CKY2018,HLP2023,LX2021}), in which the incompressible limit is somehow equivalent to  the Hele-Shaw limit.
\end{remark}

\paragraph{Outline of the paper.}\rm  
We prove Theorem~\ref{crelation}  in Section~\ref{slcr}, and Theorem~\ref{hks} is verified in  Section \ref{std}.

\section{Strong limit of pressure and complementarity relation}\label{slcr}
We first state the global existence of weak solutions of \eqref{deq}.
\begin{theorem}[Existence of global weak solutions]
Assume that the initial data $u_0$ satisfies \eqref{a1}. Then, for any $T>0$ and $m>1$, the initial boundary value  problem~\eqref{deq} and \eqref{ini}  admits at least one weak solution $(u_m,c_m)$ satisfying \[0\leq u_m,c_m\leq K,\quad\text{a.e. in }Q_T,\]\[u_m^m\in L^2(0,T;H^1(\Omega)),\quad c_m\in L^\infty(0,T;H^1\cap W^{2,p}(\Omega))\text{ with }p\in[1,\infty).\]\end{theorem}
\begin{proof}It is clear that the comparison principle gives a priori $0\leq u_m,c_m\leq K$. Then the proof of the global existence of the weak solution becomes standard. One can refer to \cite{bertozzi2010existence} for the detailed proof procedures.$\hfill\square$ 
\end{proof}
\vspace{2mm}

We use the regularity theory for the degenerate parabolic equation and  the elliptic equation, and obtain the elementary estimates as follows:
\begin{lemma}[Elementary estimates]\label{cms}
Let $(u_m, c_m)$ be the weak solution of the initial boundary value problem \eqref{deq} and \eqref{ini} under the initial assumption \eqref{a1}, it follows  
\begin{equation}\label{e1}
0\leq u_m(x,t),c_m(x,t)\leq K,\quad x\in\Omega,\quad\forall t\ \ge0,
\end{equation}
\begin{equation}\label{e3}
\|c_m\|_{L^\infty(0,\infty;W^{2,q}(\Omega))}\leq C\quad \text{for }1\leq q<\infty,
\end{equation}
\begin{equation}\label{e4}
\|\partial_t c_m\|_{L^2((0,\infty);H^{1}(\Omega))}\leq C.
\end{equation}
Moreover, there exist $u_\infty,\ \overline{u_\infty^2}\in L^\infty((0,\infty)\times\Omega)$ and $c_\infty\in L^\infty(0,\infty;W^{2,q}(\Omega))$ with $1\leq q<\infty$  such that, for any $T>0$, upon the extraction of a subsequence, as $m\to\infty$, it holds
\begin{equation}\label{e5}
u_m,\ u_m^2\rightharpoonup u_\infty,\ \overline{u_\infty^2},\quad *-\text{weak in } L^\infty(Q_T),
\end{equation}
\begin{equation}\label{e6}
c_m\to c_\infty,\quad \text{strongly in }L^p(0,T;W^{1,p}(\Omega))\text{ for }1\leq p<\infty.
\end{equation}
\end{lemma}
\begin{proof}
The first estimate~\eqref{e1} can be obtained directly via the maximum principle, and thus \eqref{e5} holds. By the $L^p$ regularity theory \cite{gn1983} for the uniform elliptic equation~$\eqref{deq}_2$, we get~\eqref{e3}. 
To obtain \eqref{e4}, we observe from \eqref{deq}$_2$ that
\begin{equation*}
-\Delta \partial_t c_m+\partial_t c_m=\partial_t u_m=\nabla\cdot (\nabla u_m^m-u_m(K-u_m)\nabla c_m)\in L^2(0,T; H^{-1}(\Omega)),
\end{equation*}
which implies
\begin{equation*}
\partial_t c_m \in L^2(0,T;H^{1}(\Omega)).
\end{equation*}
Finally, using the Aubin-Lions-Simon lemma (Lemma~\ref{al}), \eqref{e6} is valid. $\hfill\square$ 
\end{proof}

\vspace{4mm}

We further give some a priori estimates. 
\begin{lemma}\label{be}
Assume that \eqref{a1} and \eqref{a2} hold for the initial data $(u_{m,0},P_{m,0})$ and $m\geq 2$. Then, we have 

\vspace{2mm}

\noindent(\romannumeral1)\ $\int_{\Omega}u_m(t) dx\leq \alpha|\Omega|$,\quad $0\leq t\leq T$,

\vspace{2mm}

\noindent(\romannumeral2)\ $\|\nabla P_m\|_{L^2(Q_T)}\leq C(T),\quad \|\nabla u_m^{m+k}\|_{L^2(Q_T)}\leq C(T)$ for given $k\geq 0$,

\vspace{2mm}

\noindent(\romannumeral3)\ $\|P_m\|_{L^2(Q_T)}\leq C(T),\quad \| u_m^{m+k}\|_{L^2(Q_T)}\leq C(T)$ for given $k\geq 0$,

\vspace{2mm}

\noindent(\romannumeral4)\ $\|(u_m-1)_+\|_{L^2(Q_T)}\to 0$, as $m\to\infty$, where $(u_m-1)_+:=\max\{u_m-1,0\}$.

\vspace{2mm}

\noindent Moreover, there exists $P_\infty\in L^2(0,T;H^1(\Omega))$ such that for any given $k\geq0$, up to a subsequence, it holds
\begin{equation}\label{pwc}
P_m,\ u_m^{m+k}\rightharpoonup P_\infty\quad \text{ weakly in }L^2(0,T;H^1(\Omega)),\quad \text{as }m\to\infty.
\end{equation}
 
\end{lemma}
\begin{proof}
Thanks to the mass conservation and the initial condition~\eqref{a1}, then (\romannumeral1) holds.
In order to obtain the first estimate of (\romannumeral2), we integrate the pressure equation~\eqref{pe} on $Q_T$, from Lemma~\ref{cms}, it yields  
\begin{equation*}
\begin{aligned}
(m-2)\iint_{Q_T}|\nabla P_m|^2dxdt&=\int_{\Omega} P_{m,0}dx-\int_{\Omega}P_m(T)dx+(m-2)\iint_{Q_T}(K-u_m)\nabla c_m\cdot\nabla P_m dxdt\\
&\leq C+\frac{(m-2)}{2}\iint_{Q_T}(K-u_m)^2|\nabla c_m|^2dxdt+\frac{(m-2)}{2}\iint_{Q_T}|\nabla P_m|^2dxdt\\
&\leq C+(m-2)C+\frac{m-2}{2}\iint_{Q_T}|\nabla P_m|^2 dxdt,
\end{aligned}
\end{equation*}
then we directly come to the conclusion. The second estimate of (\romannumeral2) is given by 
\begin{equation*}
\iint_{Q_T}|\nabla u_m^{m+k}|^2dxdt\leq C K^{2k+2}\iint_{Q_T}|\nabla P_m|^2dxdt.
\end{equation*}
Since 
\[|\{u_m\geq 1\}|\leq \int_{\Omega}u_mdx\leq \alpha|\Omega|,\]
which implies 
\[|\{u_m<1\}|=|\Omega|-|\{u_m\geq 1\}|\geq (1-\alpha)|\Omega|.\]
It is immediate that
\[(u_m^{m-1}-1)_+=0,\quad\text{in }\{u_m<1\}.\]
Furthermore, we use Lemma~\ref{ei} with $S:=\{u_m<1\}$ and get 
\[
\begin{aligned}\iint_{Q_T}|(u_m^{m-1}-1)_+|^2dxdt\leq&C\iint_{Q_T}|\nabla(u_m^{m-1}-1)_+|^2dxdt\leq C\iint_{Q_T}|\nabla u_m^{m-1}|^2dxdt\\
\leq& C(T),\quad m>2.
\end{aligned}\]
Hence, we derive the first estimate of (\romannumeral3)
\begin{equation*}
\begin{aligned}
\iint_{Q_T} |u_m^{m-1}|^2dxdt\leq \iint_{Q_T} |(u_m^{m-1}-1)_++1|^2dxdt\leq& 2\iint_{Q_T} |(u_m^{m-1}-1)_+|^2dxdt+2T|\Omega|dxdt\\
\leq &C(T),
\end{aligned}
\end{equation*}
where we use $(a+b)^2\leq 2a^2+2b^2$. For the second estimate of (\romannumeral3), we can directly calculate 
\[\begin{aligned}
\iint_{Q_T}|u_m^{m+k}|^2dxdt\leq K^{2(1+k)}\iint_{Q_T}|u_m^{m-1}|^2dxdt\leq C(T).
\end{aligned}\]

By means of the convexity, we attain
\beq\label{341}
m(u_m-1)_+\leq u_m^m,\quad (x,t)\in Q_T.\eeq
Combining the estimates  (\romannumeral1)-(\romannumeral3) and integrating \eqref{341} on $Q_T$ in $L^2$ yield
\[\|(u_m-1)_+\|_{L^2(Q_T)}\leq \frac{1}{m}\|u_m^m\|_{L^2(Q_T)}\leq \frac{C(T)}{m},\]
which proves (\romannumeral4). 

On account of (\romannumeral2), there exists $P_\infty\in L^2(0,T; H^1(\Omega))$ such that after the extraction of a subsequence (still denoted as $P_m$), it holds 
\begin{equation*}
P_m\rightharpoonup P_\infty\quad \text{ weakly in }L^2(0,T;H^1(\Omega)),\quad \text{as }m\to\infty.
\end{equation*}
Let $Q_\infty\in L^2(0,T;H^1(\Omega))$ be the weak limit of $ u_m^{m+k}$ in $L^2(0,T;H^1(\Omega))$. We use Young's inequality and get 
\[P_m\leq \frac{k+1}{m+k}(\frac{m}{m-1})^{\frac{m+k}{k+1}}+\frac{m-1}{m+k}u_m^{m+k}.\]
Taking the weak limit as $m\to\infty$ on the above inequality means
\begin{equation}\label{c1}
P_\infty\leq Q_\infty.
\end{equation}
Conversely, for any $A>1$, we have 
\begin{equation}\label{cinequ}
\begin{aligned}
u_m^{m+k}\leq& \frac{m-1}{m}u_m^{k+1}P_m= \frac{m-1}{m}u_m^{k+1}[\min\{P_m,A\}+(P_m-A)_+]\\
\leq& \frac{m-1}{m}(\frac{m-1}{m}A)^{\frac{k+1}{m-1}}P_m+\frac{m-1}{m}K^{k+1}\frac{P_m^2}{A},
\end{aligned}
\end{equation}
where we use $(P_m-A)_+\leq \frac{P_m^2}{A}.$
We pass \eqref{cinequ} to limit in the distributional sense as $m\to\infty$ such that for some $C_\infty\in \mathcal{M}(Q_T)$, it holds  
\begin{equation}\label{c2}
Q_\infty\leq P_\infty+\frac{C_\infty}{A}.
\end{equation}
According to the arbitrariness of $A\geq1$, we conclude by \eqref{c1} and \eqref{c2} that 
\[Q_\infty=P_\infty,\quad\text{ a.e. in }Q_T,\]
and \eqref{pwc} holds.
\end{proof}

\vspace{4mm}
The absence of the derivative estimates of the density leads to a failure in the strong convergence of the density sequence $\{u_m\}_{m>1}$. However, in the following, we still prove that 
$u_m\sqrt{P_\infty}$ strongly converges to $\sqrt{P_\infty}$ in $L^2(Q_T)$ as $m\to\infty$, which is crucial for proving the strong convergence of the pressure. This is also
a new observation concerning the incompressible limit. 
\begin{lemma}\label{cr}
Let  $(u_\infty,\overline{u_\infty^2},c_\infty)$ be the limit of $(u_m,u_m^2,c_m)$.  Then, the limits $u_\infty, \overline{u_\infty^2}, c_\infty$ are bounded, namely
\begin{equation}\label{low1}0\leq u_\infty,\overline{u_\infty^2},c_\infty\leq 1,\quad\text{a.e. in } Q_T.\end{equation}
Furthermore, the  limits  satisfy
\begin{equation}\label{sr1}
u_\infty P_\infty=P_\infty,\quad u_\infty \nabla P_\infty=\nabla P_\infty,\quad\text{a.e. in } Q_T,
\end{equation}
\begin{equation}\label{sr2}
\overline{u_\infty^2} P_\infty=P_\infty,\quad \overline{u_\infty^2} \nabla P_\infty=\nabla P_\infty,\quad\text{a.e. in } Q_T,
\end{equation}
and 
\begin{equation}\label{dualitye}
\partial_t u_\infty P_\infty=0,\quad\text{ in }\mathcal{D}'(Q_T).
\end{equation}
Moreover,  up to the extraction of a subsequence, as $m\to\infty$, it holds
\begin{equation}\label{ums}
u_m\sqrt{P_\infty}\to \sqrt{P_\infty},\quad \text{ strongly in }L^2(Q_T),
\end{equation}
\begin{equation}\label{123}
u_m\sqrt{|\partial_iP_\infty|}\to \sqrt{|\partial_iP_\infty|},\quad i=1,...,n,\quad\text{ strongly in }L^2(Q_T).
\end{equation}
\end{lemma}
\begin{proof}
By direct computations, we get 
\[(u_m^2-1)_+\leq (K+1)(u_m-1)_+. \]
Then, combining this with Lemma~\ref{be} (\romannumeral4) supports the boundedness of $u_\infty,\overline{u_\infty^2}$ in \eqref{low1}. The boundedness of  $c_\infty$ can be verified by the comparison principle for the elliptic equation
\begin{equation}\label{cinfty}-\Delta c_\infty+c_\infty=u_\infty,\quad \text{a.e.}\text{ in}\ Q_T.\end{equation}

From Eq.~\eqref{deq}, we estimate
\[\|\partial_t u_m\|_{L^2(0,T;\dot{H}^{-1}(\Omega))}\leq \|\nabla u_m^m\|_{L^2(Q_T)}+\|u_m(K-u_m)\nabla c_m\|_{L^2(Q_T)}\leq C(T).\]
By means of the Sobolev embedding theorem, it follows up to a subsequence that 
\[u_m\to u_\infty,\quad \text{ strongly in }L^2(0,T;\dot{H}^{-1}(\Omega)),\ \text{ as }m\to \infty.\]
On the one hand, it holds by the weak-strong convergence that $u_mP_m$ converges to $u_\infty P_\infty$ in the sense of distribution as $m\to\infty$. On the other hand, thanks to \eqref{pwc} (Lemma \ref{be}),  $u_mP_m\rightharpoonup P_\infty\in L^2(Q_T)$ as $m\to\infty$. Thus, we have 
\[u_\infty P_\infty=P_\infty,\quad\text{a.e. in }Q_T. \]
Furthermore, as in \cite{HLP2023}, since, for any $\alpha>0$, 
\[u_\infty\nabla P_\infty^{1+\alpha}=(1+\alpha)u_\infty P_\infty^\alpha\nabla P_\infty=(1+\alpha) P_\infty^\alpha\nabla P_\infty=\nabla P_\infty^{1+\alpha}.\]
If $0<\alpha\leq \frac{1}{3}$,  we have $\nabla P_\infty^{1+\alpha}\in L^{\frac{3}{2}}(Q_T)$ and 
\[\nabla P_\infty^{\alpha+1}\rightharpoonup\nabla P_\infty,\quad \text{weakly in }L^{\frac{3}{2}}(Q_T),\ \text{as }\alpha\to0^+.\]
Therefore, it holds after $\alpha\to 0^+$ that 
\begin{equation}\label{uinfty} u_\infty \nabla P_\infty=\nabla P_\infty,\quad \text{a.e. in }Q_T.
\end{equation}
And \eqref{sr1} is proved.

Next,  we first show the strong convergence of $|u_m(1-u_m)P_\infty|$. Since 
\begin{equation*}
\begin{aligned}
\iint_{Q_T}&|u_m(1-u_m)P_\infty|dxdt\\
&=\iint_{\{u_m\leq 1\}}u_m(1-u_m)P_\infty dxdt+\iint_{\{u_m\geq 1\}}u_m(u_m-1)P_\infty dxdt\\
&\leq \iint_{\{u_m\leq 1\}}(1-u_m)P_\infty dxdt+K\iint_{\{u_m\geq 1\}}(u_m-1)P_\infty dxdt\\
&\leq \iint_{Q_T}(1-u_m)P_\infty dxdt+(K+1)\iint_{Q_T}(u_m-1)_+P_\infty dxdt.
\end{aligned}
\end{equation*}
Passing to the limit on $m$ for the above inequality, it holds by Lemma~\ref{be}  that 
\begin{equation*}
\begin{aligned}
\lim\limits_{m\to\infty}&\iint_{Q_T}|u_m(1-u_m)P_\infty|dxdt\\
&\leq \iint_{Q_T}(1-u_\infty)P_\infty dxdt+\lim\limits_{m\to\infty}(K+1)\|(u_m-1)_+\|_{L^2(Q_T)}\|P_\infty\|_{L^2(Q_T)}
=0.
\end{aligned}
\end{equation*}
Hence, we obtain 
\begin{equation*}
u_m(1-u_m)P_\infty\to0,\quad \text{ strongly in }L^1(Q_T),\ \text{as }m\to\infty.
\end{equation*}
In addition, the weak convergence in $L^2$ yields 
\begin{equation*}
u_m(1-u_m)P_\infty\rightharpoonup (u_\infty-\overline{u_\infty^2})P_\infty=0,\quad \text{ weakly in }L^2(Q_T),\ \text{as }m\to\infty.
\end{equation*}
Hence, we get 
\begin{equation*}
\overline{u_\infty^2}P_\infty=u_\infty P_\infty =P_\infty.
\end{equation*}
Similar to \eqref{uinfty}, we attain 
\begin{equation*}
\overline{u_\infty^2}\nabla P_\infty =\nabla P_\infty,
\end{equation*}
and \eqref{sr2} holds.

By passing to the limit for the density equation~\eqref{deq},  it holds 
\begin{equation}\label{deq1}\partial_t u_\infty=\Delta P_\infty-\nabla\cdot((Ku_\infty-\overline{u_\infty^2})\nabla c_\infty),\quad \text{in }\mathcal{D}'(Q_T).\end{equation}
From which, we derive 
\[\partial_t u_\infty\in L^2(0,T;\dot{H}^{-1}(\Omega)).\]
For any $h>0$, we have 
\[\frac{u_\infty(t+h)-u_\infty(t)}{h}P_\infty(t)=\frac{u_\infty(t+h)-1}{h}P_\infty(t)\leq 0,\]
\[\frac{u_\infty(t)-u_\infty(t-h)}{h}P_\infty(t)=\frac{1-u_\infty(t-h)}{h}P_\infty(t)\geq 0.\]
Let $h\to0^+$, then \eqref{dualitye} holds.

Finally, for \eqref{ums}, based on results \eqref{sr1} and \eqref{sr2}, we have
\begin{equation*}
    \iint_{Q_T}(u_m-1)^2|\partial_i P_\infty| dxdt\to\iint_{Q_T}(\overline{u_\infty^2}-2u_\infty+1)|\partial_i P_\infty| dxdt=0,\text{ as }m\to\infty,
\end{equation*} 
for $i=1,...,n$. Hence, we derive \eqref{123}.  \hfill
\end{proof}

We follow the approach of \cite{LX2021,noemi2023,HLP2023}, and give the strong convergence of the pressure.
\begin{lemma}\label{pressure}
Let $u_m$ be the weak solution to the initial boundary value problem of~\eqref{deq} and \eqref{ini}  under the initial assumptions \eqref{a1} and \eqref{a2}. Then, there exists $P_\infty\in L^2(0,T;H^1(\Omega))$ such that, after extraction of the subsequence (still denoted as $u_m^m,\ P_m$), it holds
\begin{equation}\label{umss}
\nabla u_m^m\to \nabla P_\infty,\quad\text{strongly in }L^2(Q_T),\text{ as }m\to\infty,
\end{equation}
\begin{equation}\label{um-1s}
\nabla P_m\to \nabla P_\infty,\quad\text{strongly in }L^2(Q_T),\text{ as }m\to\infty.
\end{equation}
\end{lemma}
\begin{proof}
We multiply \eqref{deq} by $u_m^{m}-P_\infty$ and integrate on $Q_T$, we obtain  \begin{equation*}
\begin{aligned}
\int&\hskip-4pt\int_{Q_T}|\nabla(u_m^m-P_\infty)|^2dxdt
=-\int_{\Omega}\frac{1}{m+1}u_m^{m+1}(t)dx|_0^T+\iint_{Q_T}\partial_t u_m P_\infty dxdt\\
&-\iint_{Q_T}\nabla P_\infty\cdot\nabla (u_m^m-P_\infty)dxdt+\iint_{Q_T}\nabla c_m\cdot\nabla (K\frac{m}{m+1}u_m^{m+1}-\frac{m}{m+2}u_m^{m+2})dxdt\\
&-\iint_{Q_T}(K-1)u_m\nabla c_m\cdot \nabla P_\infty dxdt-\iint_{Q_T}u_m(1-u_m)\nabla c_m\cdot\nabla P_\infty dxdt.
\end{aligned}
\end{equation*}

Taking the limit as $m\to\infty$, we have 
\begin{equation*}
\begin{aligned}
\lim\limits_{m\to\infty}\int&\hskip-4pt\int_{Q_T}|\nabla(u_m^m-P_\infty)|^2dxdt
=0+\iint_{Q_T}\partial_t u_\infty P_\infty dxdt-0\\
&+(K-1)\iint_{Q_T}(1-u_\infty)\nabla c_\infty \cdot\nabla P_\infty dxdt-\iint_{Q_T}(u_\infty-\overline{u_\infty^2})\nabla c_\infty \cdot\nabla P_\infty dxdt.
\end{aligned}
\end{equation*}
Hence, using the fact $\partial_t u_\infty P_\infty=0$ in the duality $L^2(H^1)\times L^2(H^{-1})$ and by means of \eqref{sr1} and \eqref{sr2} (Lemma~\ref{cr}), we conclude 
\begin{equation}\label{uums}\lim\limits_{m\to\infty}\iint_{Q_T}|\nabla(u_m^m-P_\infty)|^2dxdt=0,
\end{equation}
which proves \eqref{umss}.

Compared with \cite{HLP2023} via the cut-off technique, we can give a simpler approach to prove \eqref{um-1s} based on \eqref{umss} by making use of the weighted strong convergence \eqref{ums}. For any $\varepsilon\in (0,\frac{1}{2})$, we decompose 
\begin{equation*}
\begin{aligned}
\nabla &u_m^{m-1}-\nabla P_\infty\\
=&\nabla u_m^{m-1}\textbf{1}_{\{u_m\leq1-\varepsilon\}}+\nabla u_m^{m-1}\textbf{1}_{\{u_m\geq1-\varepsilon\}}-\nabla P_\infty\\
=&\nabla u_m^{m-1}\textbf{1}_{\{u_m\leq1-\varepsilon\}}+\frac{m-1}{mu_m}\nabla (u_m^{m}-P_\infty)\textbf{1}_{\{u_m\geq1-\varepsilon\}}-(\textbf{1}_{\{u_m\geq1-\varepsilon\}} \frac{m-1}{mu_m}-1)\nabla P_\infty.
\end{aligned}
\end{equation*}
Using strong convergence \eqref{ums} and the dominated convergence theorem, we get 
\begin{equation}\label{cfs}
\begin{aligned}
(\textbf{1}_{\{u_m\geq1-\varepsilon\}}& \frac{m-1}{mu_m}-1)\nabla P_\infty\\
&\to(\textbf{1}_{\{u_\infty\geq1-\varepsilon\}} \frac{1}{u_\infty}-1)\nabla P_\infty=0,\quad \text{ strongly in }L^2(Q_T),\ \text{as }m\to\infty.
\end{aligned}
\end{equation}
 By the triangle inequality, we directly get 
\begin{equation*}
\begin{aligned}
&\|\nabla (u^{m-1}-P_\infty)\|_{L^2(Q_T)}\leq (m-1)(1-\varepsilon)^{\frac{m-3}{2}}\|\nabla u_m\cdot\nabla u_m^m\|_{L^1(Q_T)}^2\\
&+\frac{m-1}{m (1-\varepsilon)}\|\nabla (u_m^m-P_\infty)\|_{L^2(Q_T)}+\|(\textbf{1}_{\{u_m\geq1-\varepsilon\}} \frac{m-1}{mu_m}-1)\nabla P_\infty\|_{L^2(Q_T)},
\end{aligned}
\end{equation*}
and it follows from \eqref{uums}, \eqref{cfs} and the fact $\iint_{Q_T}\nabla u_m\cdot\nabla u_m^mdxdt\leq C$ that 
\[\|\nabla (u^{m-1}-P_\infty)\|_{L^2(Q_T)}\to0,\quad \text{as }m\to\infty,\]
and \eqref{um-1s} holds.
\hfill
\end{proof}

\vspace{4mm}
 
On account of the gotten estimates in Lemmas \ref{cr}-\ref{pressure}, the \emph{complementarity relation} is verified  by 
\begin{proposition}[Complementarity relation]\label{complementarityr}Under the initial assumptions \eqref{a1}-\eqref{a2}, for any $K>0$, the  limits $(u_\infty,c_\infty,P_\infty)$ of $(u_m,c_m,P_m)$ satisfy the complementarity relation, that is
\begin{equation}\label{CR}\begin{aligned}
&P_\infty[\Delta P_\infty+(K-1)(1-c_\infty)]=0,\ &&\text{in }\mathcal{D}'(Q_T),\\
&\nabla P_\infty\cdot \vec{n}=0,\ &&\text{on }\Gamma_T,
\end{aligned}\end{equation}
where $\vec{n}$ is the outer normal vector on the boundary $\partial\Omega$. In particular, for $K=1$, 
\begin{equation}\label{39}
\nabla P_\infty=0,\quad \text{ a.e. in }Q_T,\end{equation}
and for $K<1$,
\begin{equation}\label{40}
P_\infty=0,\quad \text{a.e. in } Q_T.
\end{equation}
\end{proposition}
\begin{proof}
 For the smooth test function $\varphi\in C^1(\overline{\Omega}\times(0,T))$, it follows from Eq.~\eqref{pe} that  
\[\begin{aligned}
\frac{1}{m-1}\iint_{Q_T}P_m\partial_t \varphi dxdt=\iint_{Q_T}&|\nabla P_m|^2\varphi dxdt+\iint_{Q_T}P_m\nabla P_m\cdot\nabla\varphi dxdt\\
&-\iint_{Q_T}P_m(Ku_m-u_m^2-Kc_m+u_mc_m)\varphi dxdt\\
&-\frac{1}{m-1}\iint_{Q_T}\nabla P_m\cdot(\nabla P_m+(2u_m-K)\nabla c_m)\varphi dxdt.\end{aligned}\]
Using Lemma~\ref{cms}, Lemma~\ref{be} and Lemma~\ref{pressure}, passing to the limit as $m\to\infty$ yields 
\begin{equation*}\begin{aligned}
\iint_{Q_T}|\nabla P_\infty|^2\varphi dxdt+\iint_{Q_T}P_\infty\nabla& P_\infty\cdot\nabla\varphi dxdt\\
&-\iint_{Q_T}P_\infty(Ku_\infty-\overline{u_\infty^2}-Kc_\infty+u_\infty c_\infty)\varphi dxdt=0.   
\end{aligned}\end{equation*}
By means of Lemma~\ref{cr}, we have 
\begin{equation*}
\iint_{Q_T}|\nabla P_\infty|^2\varphi dxdt+\iint_{Q_T}P_\infty\nabla P_\infty\cdot\nabla\varphi dxdt
-\iint_{Q_T}P_\infty(K-1)(1-c_\infty)\varphi dxdt=0,
\end{equation*}
which proves $\eqref{CR}_1$.

Furthermore, if $K=1$, it follows
\[P_\infty\Delta P_\infty=0,\quad\text{in }\mathcal{D}'(Q_T).\]
Integrating  on $Q_T$ and integrating by parts imply 
\[\nabla P_\infty=0,\quad \text{in }L^2(Q_T).\]
In addition, for $K<1$, we attain 
\[P_m\leq \frac{m}{m-1}K^{m-1}\to0,\text{ as }m\to\infty, \]
which means
\[P_\infty=0,\quad \text{a.e. in }Q_T.\]
\end{proof}
\noindent\underline{\textbf{Proof of Theorem~\ref{crelation}}} We verify \eqref{ss0}-\eqref{ss1} in Lemma~\ref{cms}. Limits \eqref{ss2} and \eqref{ss3} are given in Lemma \ref{pressure} and Lemma \ref{cr}, respectively. Proposition~\ref{complementarityr} shows the complementarity relation \eqref{comrela} in the sense of distribution. Equations  $\eqref{hss1}_{1-2}$ are studied in \eqref{deq1} and \eqref{cinfty}, respectively. We prove $\eqref{hss1}_{3}$ in Lemma~\ref{cr}.

\section{Strong convergence of density}\label{std}
This section is devoted to proving the strong limit of the density $\{u_m\}_{m>1}$ for $K\leq 1$. We first write the kinetic formulation of the Keller-Segel system \eqref{deq} for all $K>0$, and then derive its limit as $m$ tends to infinity, where the kinetic formulation is inspired by the previous work~\cite{bp2009}. Since the stiff pressure effect vanishes for $K\leq 1$, we justify that the limiting kinetic function only takes values $0$~or~$1$, which implies the strong limit of the density.

\vspace{3mm}

For the weak solution $u_m$ of the problem \eqref{deq}, the corresponding kinetic function $f_m$ is defined by 
\begin{equation}\label{f_m}
f_m(t,x,\xi)=\textbf{1}_{\xi<u_m(t,x)},\quad t\in [0,\infty),\ x\in\Omega,\ \xi\in \mathbb{R_+}.
\end{equation} We state the main result of this section as follows:
\begin{proposition}\label{strong convergence of u}Let $u_m$ 
 be the weak solution of  Eq.~\eqref{deq} with \eqref{ini} under the initial assumptions~\eqref{a1}-\eqref{a2}.
For any $T>0$, there exists $f_\infty(t,x,\xi)$ satisfying $0\leq f_\infty\leq 1$ such that, up to a subsequence (still denoted as $f_m$), as $m\to \infty$, it holds     
\begin{equation}\label{f_weak}
f_m \rightharpoonup f_\infty,\quad *-\text{weak} \text{ in }L^\infty(Q_T\times\mathbb{R}_+).
\end{equation}
Then, for $K\leq1$, let $u_\infty$ be the $*-$weak limit  of $u_m$ in $L^\infty(Q_T)$ as $m\to\infty$, it follows 
\begin{equation}\label{f_def}
f_\infty=\textbf{1}_{\xi<u_\infty},\quad \text{a.e. in } Q_T\times \mathbb{R}_+.
\end{equation}
Moreover, after further extraction, the strong limit of $u_m$ holds 
\begin{equation}\label{u_def}
u_m\to u_\infty,\quad\text{strongly in } L^p(Q_T)\text{ with }p\in[1,\infty),\ \text{as }m\to\infty.
\end{equation}
\end{proposition}

\subsection{Derivation of the kinetic formulation.} 

\rm To study the limiting kinetic function $f_\infty$ and further prove \eqref{f_def}, 
we introduce the kinetic formulation corresponding to Eq.~\eqref{deq}, namely
\begin{equation}\label{kf}
\begin{aligned}
&\partial_t f_m+(\xi-c_m)g(\xi)\partial_\xi f_m+g'(\xi)\nabla_x c_m\cdot \nabla_x f_m+\nabla_x\cdot(\partial_\xi f_m\nabla u_m^m)=\partial_{\xi}M_m(t,x,\xi), 
\end{aligned}
\end{equation}
in $\mathcal{D}'(Q_T\times \mathbb{R}_+)$,
where
$$g(\xi):=\xi(K-\xi)$$ 
and  $M_m(t,x,\xi)$ is a Radon measure  on $[0,\infty)\times\Omega\times\mathbb{R}_+$ (see Lemma~\ref{radon}), given by
\[
M_m(t,x,\xi):=\delta_{\xi=u_m}\nabla u_m\cdot\nabla u_m^m\geq 0.
\]

To deduce \eqref{kf}, we introduce a convex entropy 
\[\eta(u_m):=(u_m-\xi)_+,\quad \xi\in\mathbb{R}_+.\]
By direct calculations, we get the relations
\begin{equation}\label{delta1}
\partial_\xi \eta(u_m)=-f_m(t,x,\xi),\quad \partial_\xi f_m=-\delta_{\xi=u_m}=-\partial_{u_m}f_m.
\end{equation}

Next, multiplying \eqref{deq} by $\eta'(u_m)$ and by the chain rule of derivative, it follows 
\begin{equation}\label{etaum}
\begin{aligned}
\partial_t\eta(u_m)&+\nabla\cdot(q(u_m) \nabla c_m)+(u_m-c_m)[q(u_m)-u_m(K-u_m)\eta'(u_m)]\\&-\nabla\cdot(\eta'(u_m)mu_m^{m-1}\nabla u_m)
=-\eta''(u_m)\nabla u_m\cdot\nabla u_m^m,\quad\text{in }\mathcal{D}'(Q_T\times \mathbb{R}_+),
\end{aligned}
\end{equation}
where the function $q$ is given by $q'=g'\eta'$.
Then, we differentiate the equality \eqref{etaum} with respect to $\xi$ in the distribution sense, and finally get \eqref{kf} by means of \eqref{delta1}.

\vspace{3mm}

\begin{lemma}\label{radon}
$M_m$ is a real non-negative uniformly bounded Radon measure for all $m>1$.
\end{lemma}
\begin{proof}
Multiplying \eqref{deq} by $u_m$ and integrating on $\Omega$, we get
\[\begin{aligned}
\frac{1}{2}\frac{d}{d t}\int_\Omega u_m^2\ dx+\int_\Omega \nabla u_m\cdot\nabla u_m^m\  dx
&=\int_\Omega u_m(K-u_m)\nabla c_m\cdot \nabla u_m\ dx\\
&=\int_\Omega \nabla c_m\cdot\nabla(\frac{1}{2}u_m^2-\frac{1}{3}u_m^3)\ dx\\
&=\int_\Omega (u_m-c_m)\cdot(\frac{1}{2}u_m^2-\frac{1}{3}u_m^3)\ dx\le C.
\end{aligned}\]
Once again, integrating in $(0,T)$ with respect to $t$ yields 
\begin{equation*}\label{rd}\int\hskip-4pt\int_{Q_T} \nabla u_m\cdot\nabla u_m^m\  dx\ dt\le C(T)\Longrightarrow
||M_m(t,x,\xi)||_{L^1(Q_T\times\mathbb{R}_+)}\le C(T).
\end{equation*}
\end{proof}

\subsection{Passing to the limit}

 We first give the estimates on the support property of the limiting kinetic function $f_\infty$.
\begin{lemma}\label{finfty}
 Let $f_\infty$ be the $*-$weak limit of $f_m$ in $L^\infty(Q_T\times \mathbb{R})$,  up to a subsequence,  as $m\to\infty$, then 
\begin{equation}\label{suppf}
  \text{supp}(f_\infty)\subset Q_T\times [0,1],  
\end{equation}
$$0\leq f_\infty\leq \textbf{1}_{\xi< 1},\quad\text{a.e. } (x,t,\xi)\in Q_T\times\mathbb{R}_+.$$
\end{lemma}
\begin{proof}
We first verify \eqref{suppf}. Indeed,
for any non-negative test functions $\phi_1\in C_0^\infty(Q_T)$ and $\phi_2\in C_0^\infty((1,\infty))$, we use (\romannumeral4) in Lemma~\ref{be} and get 
\[\begin{aligned}
\iiint_{Q_T\times\mathbb{R}_+}f_\infty \phi_1\phi_2 dxdtd\xi &=\lim\limits_{m\to\infty}\iiint_{Q_T\times\mathbb{R}_+}f_m \phi_1\phi_2 dxdtd\xi\\
&=\lim\limits_{m\to\infty}\iint_{Q_T} \phi_1\int_{1}^{\max\{u_m(x,t),1\}}\phi_2d\xi dxdt\\
&\leq \lim\limits_{m\to\infty}C\iint_{Q_T} \phi_1(u_m-1)_+dxdt=0.
\end{aligned}\]  
It  further holds 
\[0\leq f_\infty\leq \textbf{1}_{\xi< 1},\quad\text{a.e. }(x,t)\in Q_T,\ \xi\in\mathbb{R}_+.\] \hfill
\end{proof}

\vspace{3mm}

As  $P_\infty(1-\rho_\infty)=0$ in Lemma~\ref{cr}, we describe the similar results for the limiting kinetic variable $f_\infty$ to the limiting pressure $P_\infty$. 
\begin{lemma}\label{finftyp} Assume that 
 $P_\infty\in L^2(0,T;H^1(\Omega))$ and  $(1-u_\infty)P_\infty=0$ (Lemma~\ref{cr}) hold. Then, the limiting kinetic variable $f_\infty$ (Lemma~\ref{finfty}) satisfies
\[(f_\infty-\textbf{1}_{\xi<1})P_\infty=0,\quad (f_\infty -\textbf{1}_{\xi< 1})\textbf{1}_{\mathcal{P}(t)}=0,\quad (f_\infty -\textbf{1}_{\xi<1})\nabla P_\infty=0,\quad\text{a.e. } (x,t,\xi)\in Q_T\times\mathbb{R}_+.\]
where the domain $\mathcal{P}(t)$ is defined by \eqref{support}.
\end{lemma}

\begin{proof}
For the case $0<K<1$, $P_\infty=0$, the results evidently hold. The proof of the case $K\geq1 $ is  divided into three steps.

\noindent\emph{Step 1.} Let $P_{\infty,R}=\min\{P_\infty, R\}$ for any $R>0$, up to a subsequence, it concludes from Lemma~\ref{cr} that
\begin{equation*}
u_mP_{\infty,R}\to P_{\infty,R},\quad \text{a.e. }(x,t)\in Q_T,\text{ as }m\to\infty.
\end{equation*}
To begin with, we show 
\[\int_{0}^K f_\infty(x,t,\xi)d\xi P_{\infty,R}=P_{\infty,R},\quad \text{a.e. in }Q_T.\]
Since 
\[u_m(x,t)P_{\infty,R}(x,t)=\int_{0}^K f_m(x,t,\xi)d\xi P_{\infty,R}(x,t),\quad \text{ almost everywhere in }Q_T. \]
We take the weak limit on the above equality, and get 
\[u_\infty P_{\infty,R}(x,t)=P_{\infty,R}(x,t)=\int_0^K f_\infty(x,t,\xi)d\xi P_{\infty,R}(x,t),\quad \text{a.e. }(x,t)\in Q_T.\]

\noindent\emph{Step 2.} We use
\begin{equation*}
\begin{aligned}
 \int_0^K\textbf{1}_{\xi< 1} d\xi P_{\infty,R}=P_{\infty,R},
\end{aligned}
\end{equation*}
and conclude
\begin{equation*}
    \int_0^Kf_\infty(x,t,\xi)d\xi P_{\infty,R}=\int_0^K\textbf{1}_{\xi< 1} d\xi P_{\infty,R},\quad \text{in  }\mathcal{D}'(Q_T).
\end{equation*}
Thanks to the inequality $0\leq f_\infty \leq \textbf{1}_{\xi<1}$ (Lemma~\ref{finfty}), 
it concludes 
$$(f_\infty-\textbf{1}_{\xi< 1})P_{\infty,R}=0,\quad \text{a.e. } (x,t,\xi)\in Q_T\times\mathbb{R}_+.$$
For any $\alpha>0$, we have 
$$(f_\infty-\textbf{1}_{\xi< 1})P_{\infty,R}^{\alpha}=0,\quad \text{a.e. }(x,t,\xi)\in Q_T\times\mathbb{R}_+.$$
Let $\alpha\to 0^+$, then we obtain 
$$(f_\infty-\textbf{1}_{\xi< 1}) \textbf{1}_{\mathcal{P}(t)}=0,\quad \text{a.e. }(x,t,\xi)\in Q_T\times\mathbb{R}_+.$$
\emph{Setp 3.} For $\alpha>1$, it follows 
$$f_\infty\nabla P_{\infty,R}^{\alpha}=\alpha f_\infty P_{\infty,R}^{\alpha-1}\nabla P_{\infty,R}=\textbf{1}_{\xi< 1} \nabla P_{\infty,R}^\alpha,\quad \text{a.e. }(x,t,\xi)\in Q_T\times\mathbb{R}_+.$$
After taking $\alpha\to1^+$, we get 
$$(f_\infty-\textbf{1}_{\xi< 1})\nabla P_{\infty,R}=0,\quad \text{a.e. }(x,t,\xi)\in Q_T\times\mathbb{R}_+.$$
Finally, let $R\to+\infty$, it holds 
\[f_\infty  P_\infty=\textbf{1}_{\xi<1}P_\infty,\quad f_\infty  \nabla P_\infty=\textbf{1}_{\xi<1}\nabla P_\infty,\quad\text{a.e. } (x,t,\xi)\in Q_T\times\mathbb{R}_+.\]
\end{proof}

To take the limit on the kinetic formulation~\eqref{kf} as $m\to\infty$, we recall the weak convergence $f_\infty$ from the bound of $f_m$. In addition, $M_m(t,x,\xi)=0$ also holds for  $\xi>K$. Hence, for $f_\infty=f_\infty(t,x,\xi)\in [0,1]$, and a non-negative measure $M_\infty=M_\infty(t,x,\xi)$, there exists subsequences, still denoted by $f_m$ and $M_m$, as $m\to\infty$  such that
\begin{equation}\label{fweak}
  f_m\rightharpoonup f_\infty,\quad *-\text{weak in } L^\infty(Q_T\times \mathbb{R}_+),  
\end{equation}
\[M_m\rightharpoonup M_\infty,\quad \text{weak in }\mathcal{M}^1(Q_T).\]

On the kinetic formulation \eqref{kf}, we study the  limits of  its nonlinear parts $\nabla_x\cdot(\partial_\xi f_m\nabla u_m^m)$ and $g'(\xi)\nabla_x c_m\cdot \nabla_x f_m$ respectively in the following.

\paragraph{For $\nabla_x\cdot(\partial_\xi f_m\nabla u_m^m)$.}

For any smooth test function  $\varphi=\varphi(x,t,\xi)\in C_0^\infty(Q_T\times \mathbb{R}_+)$, by the strong convergence of $\nabla u_m^m$ (Lemma~\ref{pressure}), we have 
\begin{equation*}\label{f1}
\begin{aligned}
\iiint_{Q_T\times \mathbb{R}_+}\nabla\cdot(\partial_\xi f_m\nabla u_m^m)\varphi dxdtd\xi&=\iiint_{Q_T\times \mathbb{R}_+} f_m\nabla u_m^m\cdot\nabla\partial_\xi \varphi dxdtd\xi\\
&\to \iiint_{Q_T\times \mathbb{R}_+} f_\infty \nabla P_\infty \cdot\nabla\partial_\xi \varphi dxdtd\xi,\ \text{ as }m\to\infty.
\end{aligned}
\end{equation*}
We use the equalities $f_\infty\nabla P_\infty=\textbf{1}_{\xi<1}\nabla P_\infty$ (Lemma~\ref{finftyp}) and $\partial_\xi(\textbf{1}_{\xi<1})=-\delta_{\xi=1}$,
it follows from the above formula that 
\begin{equation}\label{limit1}\nabla(\partial_\xi f_m\nabla u^m_m)\rightharpoonup- \delta_{\xi=1}\Delta P_\infty,\text{ in }\mathcal{D}'(Q_T\times \mathbb{R}_+),\text{ as }m\to\infty. \end{equation}

\paragraph{For $g'(\xi)\nabla_x c_m\cdot \nabla_x f_m$.}  
We directly calculate
\[\begin{aligned}
g'(\xi)\nabla_x c_m\cdot \nabla_x f_m
&=\nabla\cdot(g'(\xi)\nabla_x c_m\ f_m) 
-g'(\xi)\Delta_x c_m\ f_m\\
&=\nabla\cdot(g'(\xi)\nabla_x c_m\ f_m) 
+(u_m-c_m)g'(\xi)f_m.
\end{aligned}\]
Again, based on the weak-strong products limit, in the sense of distribution, as $m\to\infty$, it holds in the sense of distribution that 
\[\begin{aligned}
\nabla\cdot(g'(\xi)\nabla_x c_m\ f_m) &\rightharpoonup \nabla\cdot(g'(\xi)\nabla_x c_\infty\ f_\infty),\\
c_mg'(\xi)f_m&\rightharpoonup c_\infty g'(\xi)f_\infty.
\end{aligned}\]
Notice that the remaining term $u_mg'(\xi)f_m$ is the difficulty due to the absence of strong limit. To deal with it, we recall that  $||u_mf_m||_{L^\infty}$ is bounded, so that there exists a limit function $\rho_\infty=\rho_\infty(t,x,\xi)\in [0,1]$ such that
\begin{equation}\label{roh_inf}
u_mf_m\rightharpoonup \rho_\infty,\quad *-\text{weak in } L^\infty(Q_T\times \mathbb{R}_+),\quad\text{as }m\to\infty.
\end{equation}
Then, we take the limit as $m\to\infty$ and have
\begin{equation}\label{limit2}
g'(\xi)\nabla_x c_m\cdot \nabla_x f_m\rightharpoonup
g'(\xi)\nabla_x c_\infty\cdot \nabla_x f_\infty
+g'(\xi)(\rho_\infty-c_\infty f_\infty),\quad\text{in }\mathcal{D}'(Q_T\times\mathbb{R}_+).
\end{equation}

In addition, the limit $\rho_\infty$ can be identified as follows. For any given $\xi\geq 0$, we define a function
\[S(z):=z\cdot\textbf{1}_{0\le\xi\le z}, \quad z\in[0,+\infty).\]
Thanks to the result on weak convergence  \eqref{f_weak}, we have
\[
S(u_m)=\int_0^{u_m}S'(\eta)\ d\eta=\int_0^\infty S'(\eta)f_m\ d\eta\rightharpoonup \int_0^\infty S'(\eta)f_\infty(t,x,\eta)\ d\eta,\quad\text{as }m\to \infty.
\]
Using \eqref{roh_inf}, we find that $u_\infty$ satisfies
\[
S(u_m)=u_mf_m\rightharpoonup \rho_\infty=\int_0^\infty (\textbf{1}_{0\le\xi\le \eta}+\eta\delta_{\xi=\eta})f_\infty(t,x,\eta)\ d\eta,\quad\text{as }m\to \infty.
\]
Consequently, we get
\begin{equation}\label{rho_}
\rho_\infty(t,x,\xi)-\xi f_\infty(t,x,\xi)=\int_\xi^\infty f_\infty(t,x,\eta)\ d\eta.
\end{equation}

The following lemma of estimating $\rho_\infty-u_\infty f_\infty$ is given in \cite[Lemma 3.1]{bp2009}, and we omit the details.
\begin{lemma}\label{proved}
For any given $T>0$, and $\rho_\infty$ is given in \eqref{rho_},  it holds for some $C>0$ that 
\begin{equation*}\label{inequality}
|\rho_\infty(t,x,\xi)-u_\infty(t,x)f_\infty(t,x,\xi)|\le Cf_\infty(t,x,\xi)(1-f_\infty(t,x,\xi)),\quad \text{ a.e. } (x,t,\xi)\in Q_T\times \mathbb{R}_+. 
\end{equation*}
\end{lemma}

\vspace{3mm}

\noindent{\bf Passing to the limit on \eqref{kf}.} For $K>0$, combining the convergence results \eqref{limit1},\eqref{limit2}, we take $m\to\infty$ in \eqref{kf}, it holds in the distribution sense that
\begin{equation}\label{pass}
 \partial_t f_\infty+(\xi-c_\infty)g(\xi)\partial_\xi f_\infty +g'(\xi)\nabla_x c_\infty\cdot \nabla f_\infty
 +g'(\xi)(\rho_\infty-u_\infty f_\infty)=\delta_{\xi=1}\Delta P_\infty+\partial_{\xi}M_\infty(t,x,\xi).
\end{equation}
In particular, for $K\leq 1$, the vanishing  of the stiff pressure effect (Proposition~\ref{complementarityr}) yields 
\begin{equation*}
 \partial_t f_\infty+(\xi-c_\infty)g(\xi)\partial_\xi f_\infty +g'(\xi)\nabla_x c_\infty\cdot \nabla f_\infty
 +g'(\xi)(\rho_\infty-u_\infty f_\infty)=\partial_{\xi}M_\infty(t,x,\xi).
\end{equation*}
 The functions $f_\infty$ and $M_\infty$ in the above equation
have the following properties
\[
0\le f_\infty\le 1\quad \text{a.e. in }Q_T\times\mathbb{R}_+,\quad 
f_\infty=0\quad\text{when }\xi>K,
\]
\[
M_\infty (t,x,\xi)=0\quad\text{when }\xi>K,\quad
\iiint_{Q_T\times\mathbb{R}_+}M_\infty(t,x,\xi) \ dt\ dx\ d\xi<\infty.
\]
From \eqref{delta1}, there is a probability measure $\nu(t,x,\xi)\geq0$ such that
\begin{equation}\label{finftyd}
\partial_\xi f_\infty=-\nu(t,x,\xi)\le 0,\quad\text{in }\mathcal{D}'(Q_T\times \mathbb{R}_+).
\end{equation}

\subsection{Strong convergence of density.} 

To verify the strong convergence of $\{u_m\}_{m>0}$ for $K\leq1$, based on the kinetic theory~\cite{bp2002,bp2009}, which is equivalent to that the limiting kinetic function $f_\infty$ is a characteristic function, i.e., $f_\infty$ takes the values $0$ or $1$. To this end, the regularized technique is used, we take $\epsilon>0$ arbitrarily, and $\vp^1\in\mathcal{C}^\infty(\mathbb{R})$, $\vp^2\in\mathcal{C}^\infty(\mathbb{R}^{n+1})$. Define the mollifiers 
\begin{equation*}\label{vp_ep}
\varphi_{\epsilon}^1(\xi)=\frac{1}{\epsilon}\varphi^1(\frac{\xi}{\epsilon})\geq 0,\qquad
\varphi^2_{\epsilon}(x,t)=\frac{1}{\epsilon^{n+1}}\varphi^2(\frac{t}{\epsilon},\frac{x}{\epsilon})\geq 0,
\end{equation*}
with $\text{supp}(\varphi_\epsilon^{1})\subset [0,\epsilon)$ and $\text{supp}(\varphi_{\epsilon}^2)\subset B_\epsilon(0)\times(0,\epsilon)$, $\int_{\mathbb{R}_+}\varphi_{\epsilon}^1d\xi=1$ and $\int_{\mathbb{R}^{n+1}}\varphi_{\epsilon}^2dxdt=1$. 
For any given function $f:\mathbb{R}^n\times\mathbb{R}_+\times\mathbb{R}_+\to\mathbb{R}$, we define
\begin{equation*}\label{f_epslon}
f_\epsilon(x,t,\xi)=f\ast\varphi^1_{\epsilon}\ast\varphi_{\epsilon}^2=\int_{\mathbb{R}^n\times\mathbb{R}_+\times\mathbb{R}_+}f(z,\tau,\zeta)\varphi_{\epsilon}^2(x-z,t-\tau)\varphi_{\epsilon}^1(\xi-\zeta)dzd\tau d\zeta.
\end{equation*}

We give the interior domain 
\[\Omega_\epsilon:=\{x\in\Omega: d(x,\partial \Omega)> \epsilon\},\]
and set \[Q_{T,\epsilon}:=\Omega_\epsilon\times[0,T].\]

\vspace{3mm}

Next, we follow \cite{bp2009} and give the proof of Propsoition \ref{strong convergence of u}.
\vspace{3mm}

\noindent\underline{\textbf{Proof of Proposition \ref{strong convergence of u}.}}
We have proved \eqref{f_weak} in \eqref{fweak}.

\vspace{3mm}

\noindent\emph{Step 1, proof of $\eqref{f_def}$.}
We use 
the technique introduced in   \cite{bp2009} by comparing $f_\infty$ and $f_\infty^2$ to prove that $f_\infty$ only takes the values $0$ or $1$ almost everywhere. 
Since $0\le f_{\infty,\epsilon}\le 1$, we have 
$$0\le f_{\infty,\epsilon}-f_{\infty,\epsilon}^2\le 1,\quad -1\leq 1-2f_{\infty,\epsilon}\leq 1.$$
After regularization for \eqref{pass}, $F_\epsilon:=f_{\infty,\epsilon}(1-f_{\infty,\epsilon})\in [0,1]$ is a solution of
\begin{equation}\label{passs}
\begin{aligned}
\partial_t F_\epsilon&+(\xi-c_\infty)g(\xi)\partial_\xi F_\epsilon +g'(\xi)\nabla_x c_\infty\cdot \nabla F_\epsilon
+R_\epsilon=\partial_{\xi}M_{\infty,\epsilon}(1-2f_{\infty,\epsilon})+r_\epsilon,
\end{aligned}
\end{equation}
where $R_{\epsilon}:=\big(g'(\xi)(\rho_\infty-u_\infty f_\infty)\big)_\epsilon$ with  $F_0:=f_{\infty}(1-f_{\infty})$ and the remainder $r_\epsilon$ is
\[\begin{aligned}
r_\epsilon=(\xi-c_\infty)g(\xi)\partial_\xi F_\epsilon-((\xi-c_\infty)g(\xi)\partial_\xi F_0)_\epsilon+g'(\xi)\nabla_x c_\infty\cdot \nabla F_\epsilon-(g'(\xi)\nabla_x c_\infty\cdot \nabla F_0)_\epsilon.
\end{aligned}
\]Integrating \eqref{passs} by part on $\Omega_{\epsilon}\times\mathbb{R}_+$, and we get
\begin{equation*}\begin{aligned}
\frac{d}{dt}\iint_{\Omega_{\epsilon}\times\mathbb{R}_+}&F_\epsilon dxd\xi=-\iint_{\Omega_{\epsilon}\times\mathbb{R}_+}R_\epsilon (1-2f_{\infty,\epsilon})dxd\xi
+\iint_{\Omega_{\epsilon}\times\mathbb{R}_+}F_\epsilon\{\partial_\xi[(\xi-c_\infty)g(\xi)]\\
&+(c_\infty-u_\infty) g'(\xi)\}dxd\xi+
2\iint_{\Omega_{\epsilon}\times\mathbb{R}_+}M_{\infty,\epsilon}\partial_\xi f_\epsilon dxd\xi+\iint_{\Omega_{\epsilon}\times\mathbb{R}_+} r_\epsilon (1-2f_{\infty,\epsilon})dxd\xi\\
&-\iint_{\partial \Omega_{\epsilon}\times\mathbb{R}_+}F_\epsilon g'(\xi)\nabla c_\infty\cdot n_{\Omega_{\epsilon}}(x)dS(x) d\xi.
\end{aligned}\end{equation*}
Combining $0\leq u_\infty,c_\infty\leq 1$ (Lemma~\ref{cr}) and the definition $g(\xi)=\xi(K-\xi)$ means 
$$|\partial_\xi[(\xi-c_\infty)g(\xi)]+\Delta c_\infty g'(\xi)|\leq C,\quad \forall\xi\in[0,K+1].$$ 
We take \eqref{finftyd} into consideration, and get 
\begin{equation*}
2\iint_{\Omega_{\epsilon}\times\mathbb{R}_+}M_{\infty,\epsilon}\partial_\xi f_\epsilon dxd\xi\leq -2\iint_{\Omega_{\epsilon}\times\mathbb{R}_+}M_{\infty,\epsilon}\nu\ast\varphi_\epsilon^1\ast\varphi_\epsilon^2 dxd\xi\leq 0.
\end{equation*}
Because $f_\infty-f_\infty^2$ is concave on $f_\infty$, we use Lemma~\ref{proved} and Jensen's inequality, and then attain 
\[\iint_{\Omega_{\epsilon}\times\mathbb{R}_+}R_\epsilon (1-2f_{\infty,\epsilon})dxd\xi\leq C\iint_{\Omega_{\epsilon}\times\mathbb{R}_+}(f_\infty-f_\infty^2)_\epsilon dxd\xi\leq C\iint_{\Omega_{\epsilon}\times\mathbb{R}_+}F_\epsilon dxd\xi.\]
Hence, we use the above facts and obtain 
\begin{equation*}\begin{aligned}
\frac{d}{dt}\iint_{\Omega_{\epsilon}\times\mathbb{R}_+}F_\epsilon dxd\xi
\le &\underbrace{\|r_\epsilon(t)\|_{L^1(\Omega\times(0,2))}+C\int_{\partial \Omega_{\epsilon}}|\nabla c_\infty\cdot \vec{n}_{\Omega_{\epsilon}}(x)|dS(x)}_{=:U_\epsilon(t)}\\
&+C\iint_{ \Omega_{\epsilon}\times\mathbb{R}_+}F_\epsilon dxd\xi.
\end{aligned}\end{equation*}
By the Gr{\"o}nwall inequality, we get
\begin{equation*}
\begin{aligned}
\iint_{\Omega_{\epsilon}\times\mathbb{R}_+}F_\epsilon \le &e^{Ct}\iint_{\Omega_{\epsilon}\times\mathbb{R}_+}F_\epsilon(t=0)+\int_0^t e^{C(t-s)}U_\epsilon(s) ds.
\end{aligned}
\end{equation*}
In \cite{bp2009}, it has been proved that $r_\epsilon\to0$ in $L^1(
(0,T)\times\Omega\times(0,C)$ for all $T>0$, $C>0$. 
It is true that $\int_{\partial\Omega_{\epsilon}}|\nabla c_\infty\cdot \vec{n}_{\Omega_{\epsilon}}(x)|dS(x)\to0$ as $\epsilon\to0$ because $c_\infty\in W^{2,p}$ for $1\le q<\infty$, $\nabla c_\infty\in \mathcal{C}^{0,\alpha}$ for some $0<\alpha<1$, thus we have $$U_\epsilon\to0, \text{ as }\epsilon\to0.$$ 
Since $\|u_{m,0}-u_{\infty,0}\|_{L^1(\Omega)}\to0$ as $m\to\infty$, we have $f_{\infty}(0,x,\xi)=\textbf{1}_{\xi\leq u_{\infty,0}}$.  
Then, it only remains to prove that $F_\epsilon(t=0)$ converges to $0$. Thanks to \cite[Lemma 4.2.2]{bp2002},  it holds true that  $f_{\infty,\epsilon}(t=0)$ strongly converge to $f_\infty(t=0)=\textbf{1}_{\xi<u_{\infty,0}}$ as $\epsilon\to0$. Since $F_\epsilon\to F_0$ in $L^1_{loc}(0,\infty;L^1(\Omega\times\mathbb{R}_+))$. We deduce that $F_0=0$, so $f_\infty=f_\infty^2$, which implies $f_\infty$ takes $0$ or $1$ for almost everywhere. Since $f_\infty$ is decreasing in $\xi$, it is clear that
\[
f_\infty=\textbf{1}_{\xi<u_\infty(t,x)},
\]
which proves \eqref{f_def}.

\vspace{3mm}

\noindent\emph{Step 2, proof of $\eqref{u_def}$.}
To show \eqref{u_def}, let $S$ be any  convex function with $S'\in L^\infty$ and $S(0)=0$. From~\eqref{f_m}, we know
\begin{equation*}\label{kk1}\begin{aligned}
  \text{w-}\lim S(u_m(t,x))&=\text{w-}\lim\int_{\mathbb{R}_+}S'(\xi)f_m(t,x,\xi)d\xi=\int_{\mathbb{R}_+}S'(\xi)f_\infty(x,\xi)\ d\xi\\&=\int_{\mathbb{R}_+}S'(\xi)\textbf{1}_{\xi< u_\infty(x,t)}\ d\xi=S(u_\infty),
\end{aligned}
\end{equation*}
which means 
\begin{equation*}
u_m\to u_\infty, \text{ a.e. }(x,t)\in Q_T,\ \text{ as   }m\to\infty.
\end{equation*} 
Hence, via the dominated convergence theorem, it follows  
\begin{equation*}
u_m\to u_\infty\quad \text{ strongly in } L^p(Q_T),\ 1\leq p<\infty,\ \text{as }m\to\infty.
\end{equation*}

\vspace{3mm}

\noindent \underline{\bf Proof of Theorem~\ref{hks}}.  The property \eqref{pressurev} can be obtained from Proposition~\ref{complementarityr}. The strong limit of density \eqref{ds} is justified in Propsoition~\ref{strong convergence of u}. Using these and \eqref{hss1}, we get \eqref{hss2}. 


\subsection{Numerical illustrations}
For the solutions $u_m$ satisfying system \eqref{deq} and $u_\infty$ in \eqref{hss1}, we show the time evolution with different exponents $m$. Then, we compare the numerical results with the previously presented analysis.
\begin{figure}[ht]
\centering
\subfigure[$m=2$, $t=0$]{
\includegraphics[width=2.7cm,height=3cm]{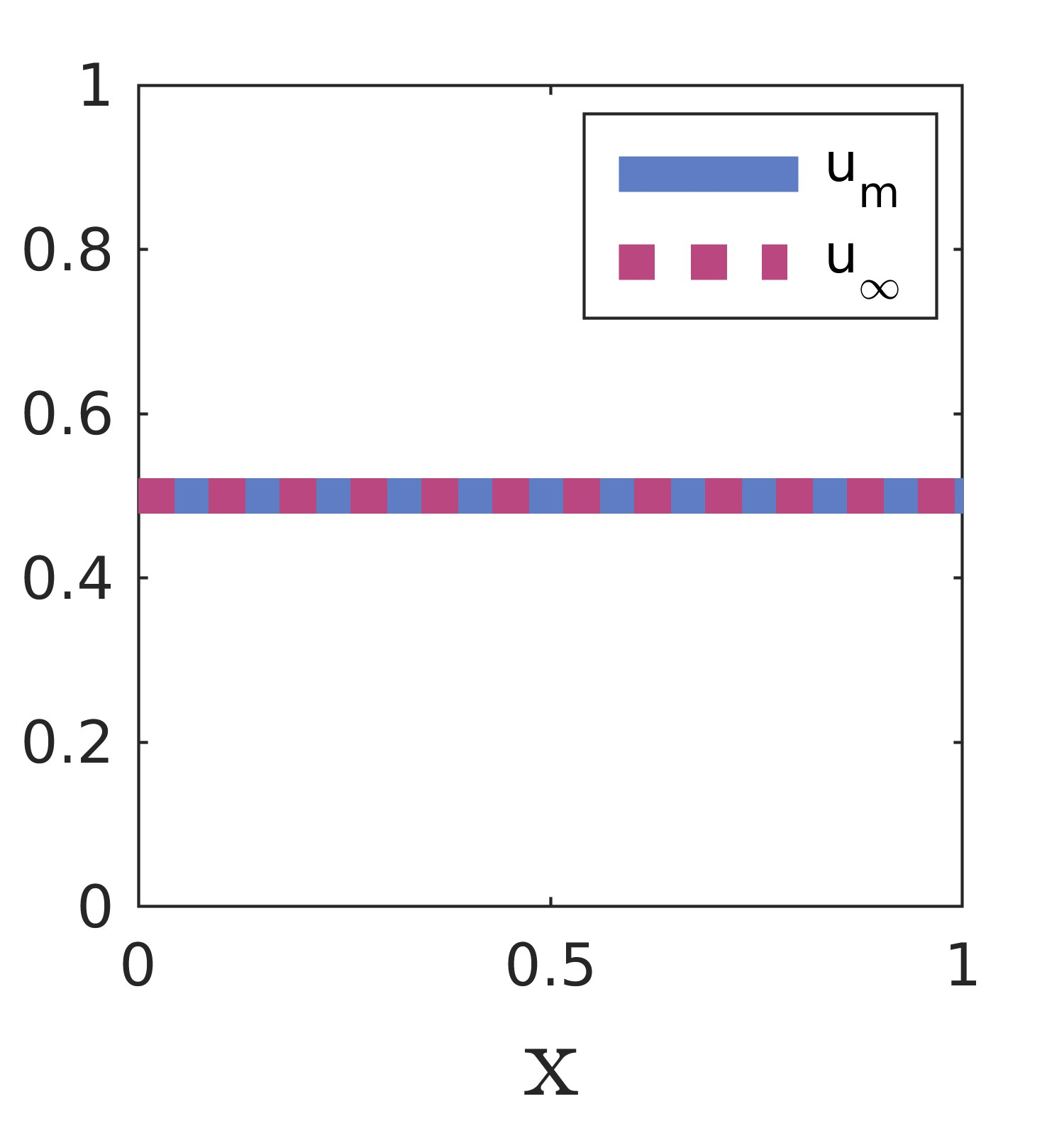}
}\hspace{0.5cm}\subfigure[$m=2$, $t=5$]{
\includegraphics[width=2.7cm,height=3cm]{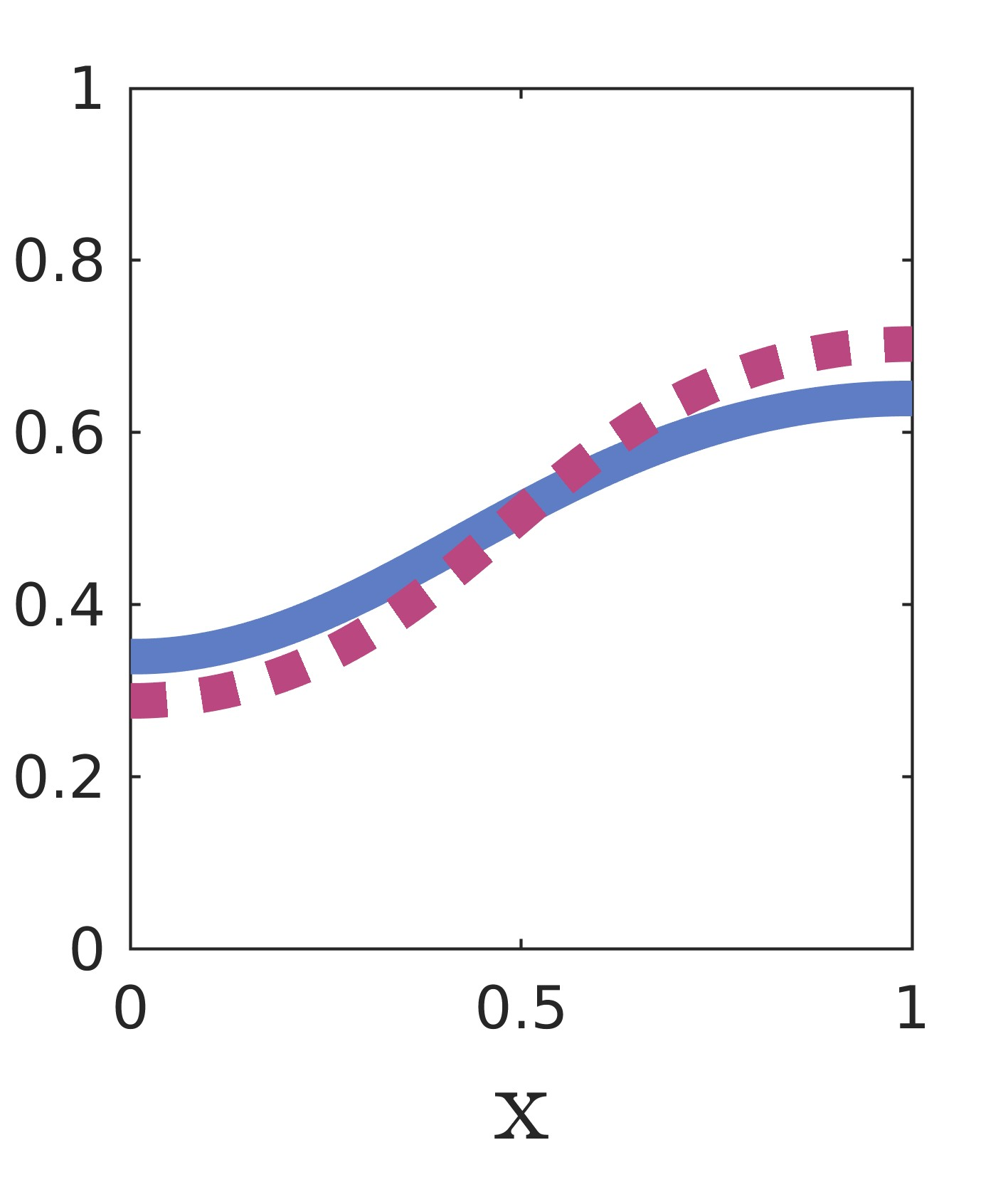}
}\hspace{0.5cm}\subfigure[$m=2$, $t=20$]{
\includegraphics[width=2.7cm,height=3cm]{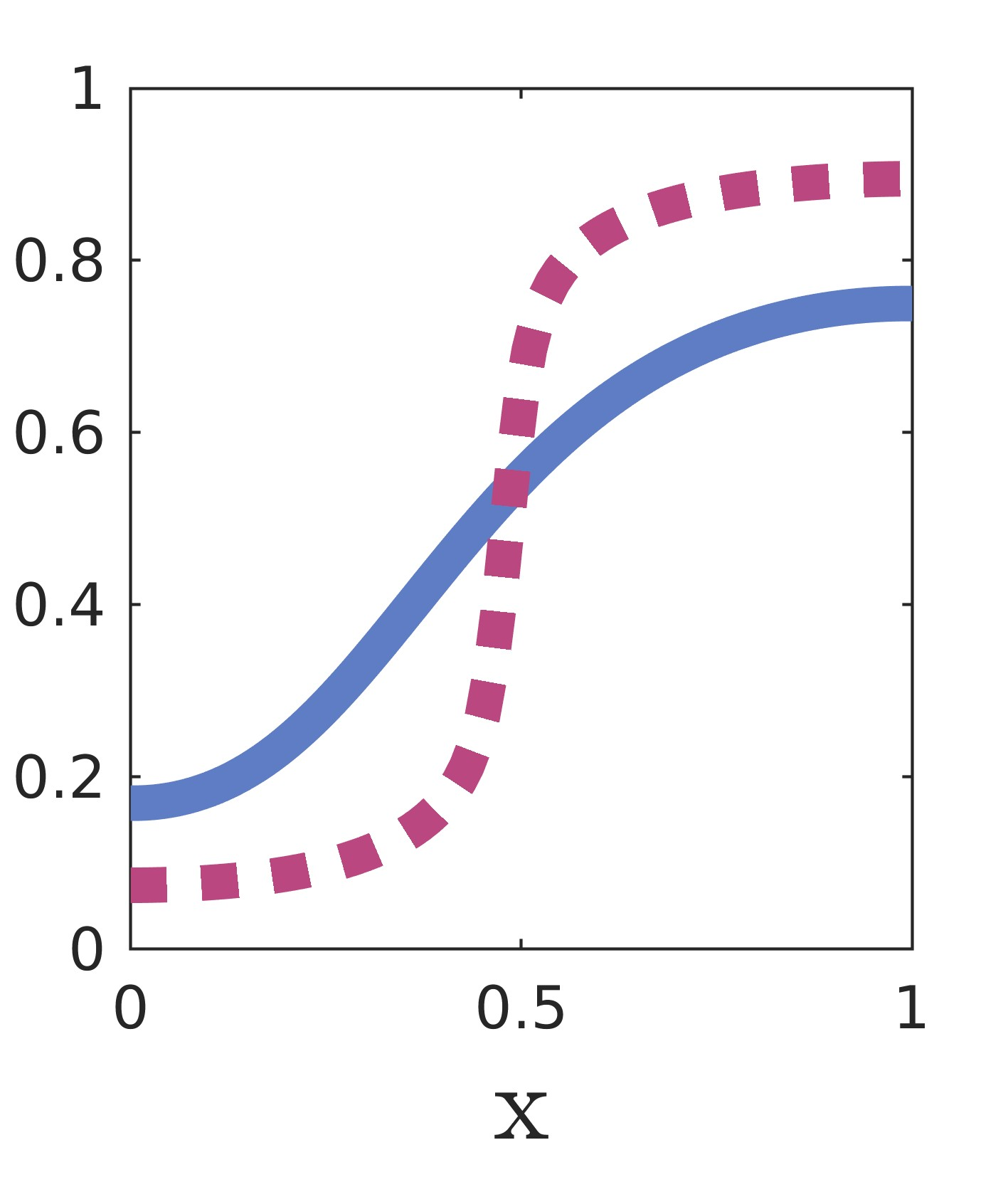}
}\hspace{0.5cm}\subfigure[$m=2$, $t=1000$]{
\includegraphics[width=2.7cm,height=3cm]{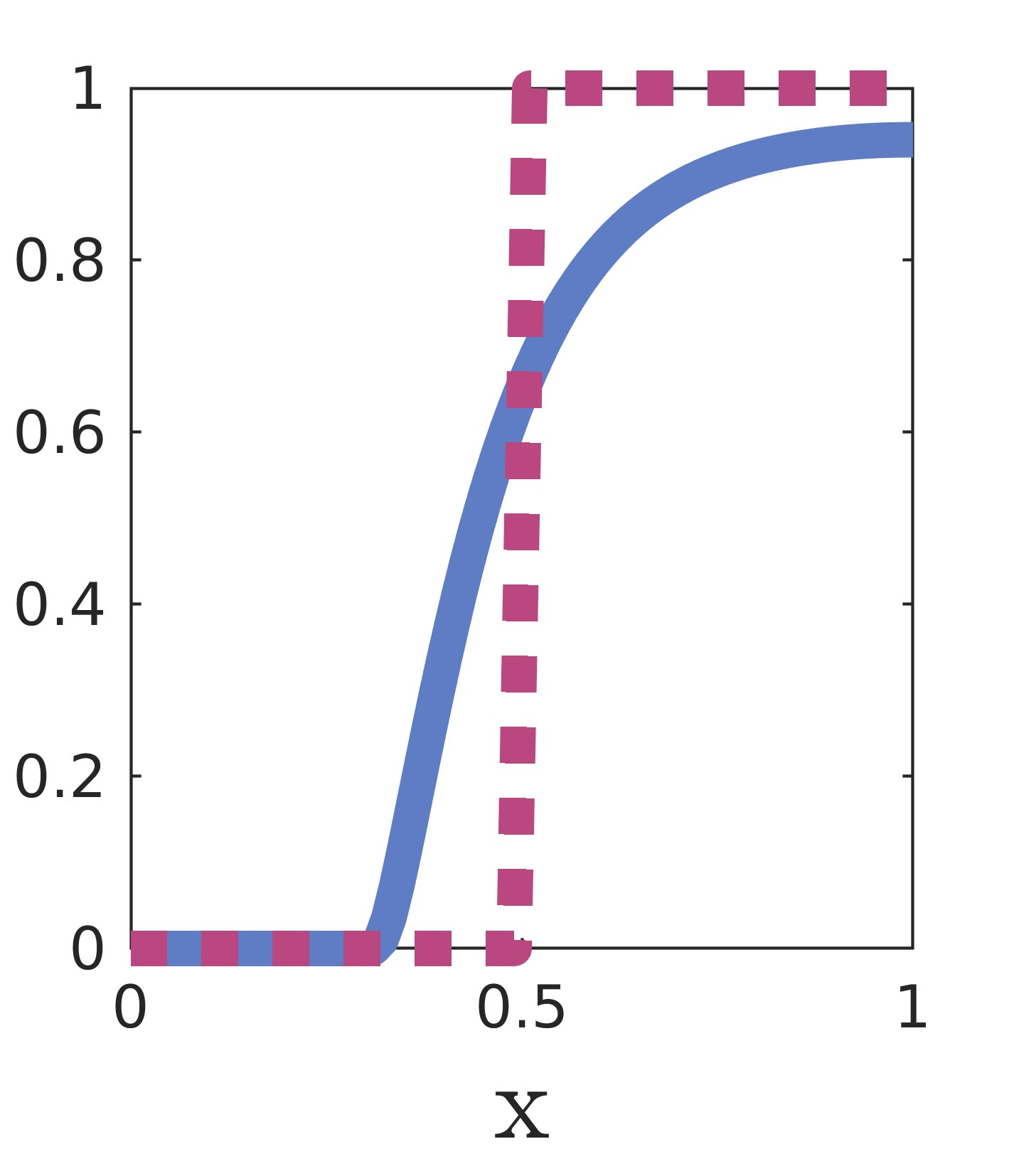}
}
\\
\subfigure[$m=5$, $t=0$]{
\includegraphics[width=2.7cm,height=3cm]{Figures/uinfty1.png}
}\hspace{0.5cm}\subfigure[$m=5$, $t=5$]{
\includegraphics[width=2.7cm,height=3cm]{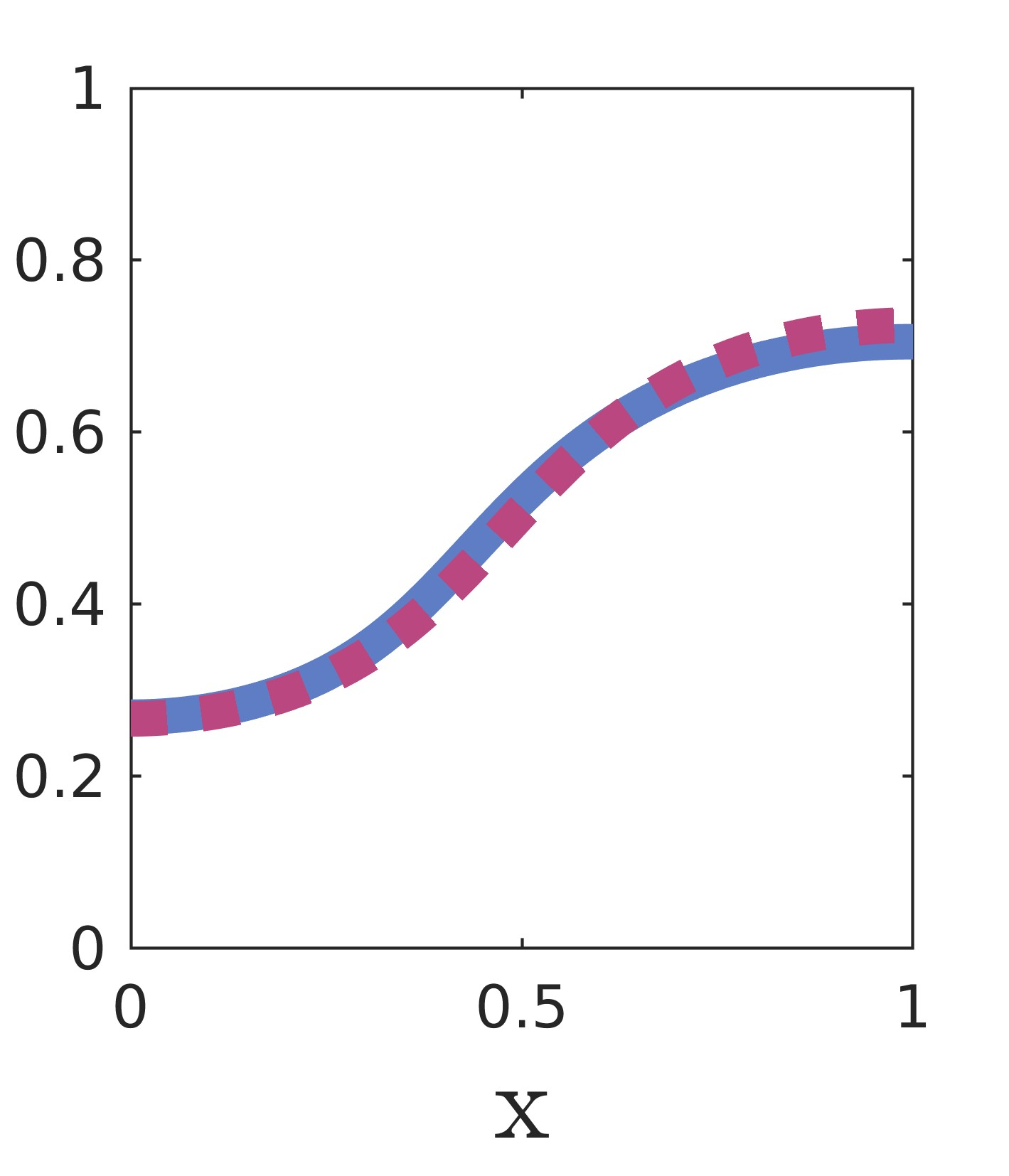}
}\hspace{0.5cm}\subfigure[$m=5$, $t=20$]{
\includegraphics[width=2.7cm,height=3cm]{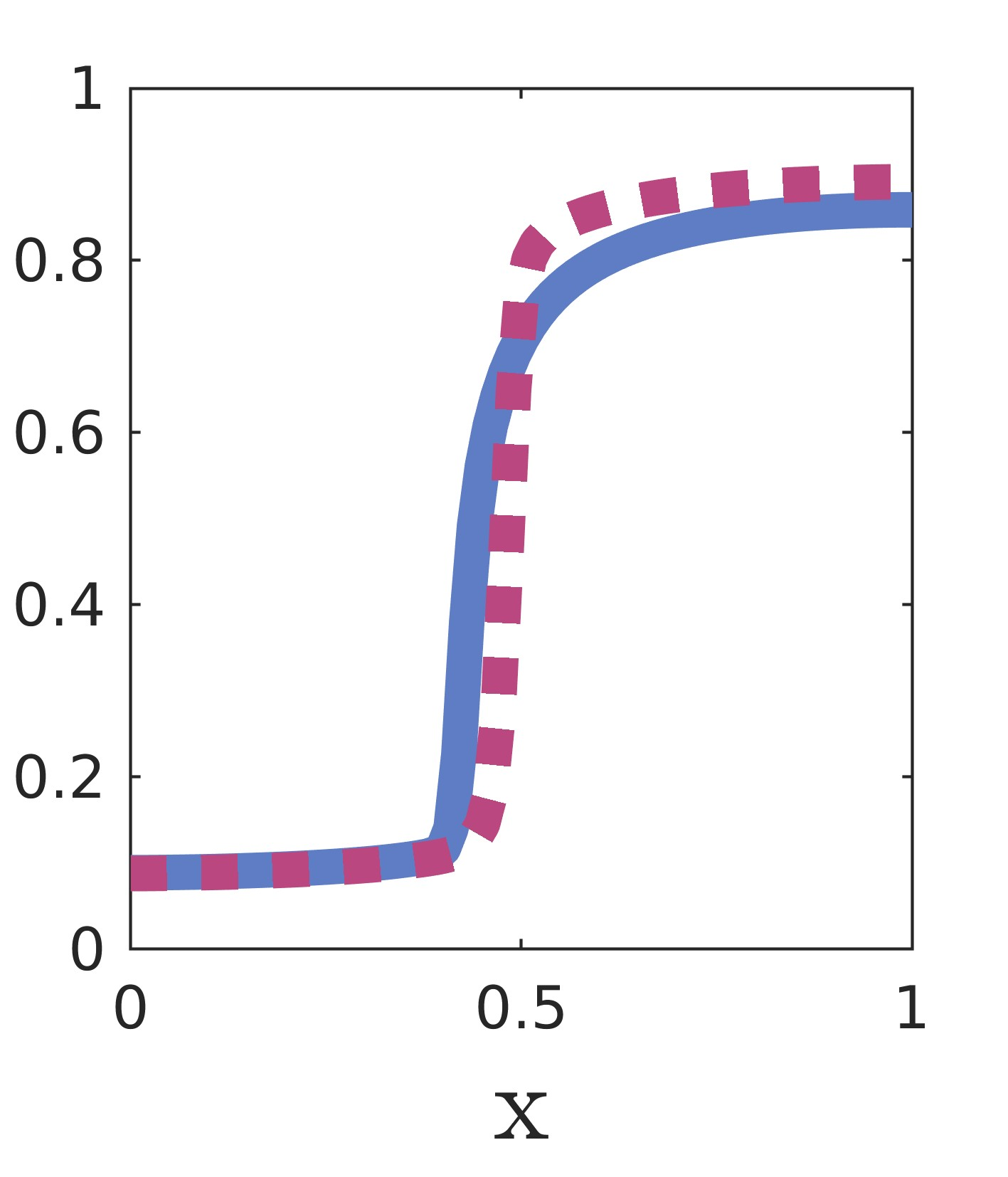}
}\hspace{0.5cm}\subfigure[$m=5$, $t=1000$]{
\includegraphics[width=2.7cm,height=3cm]{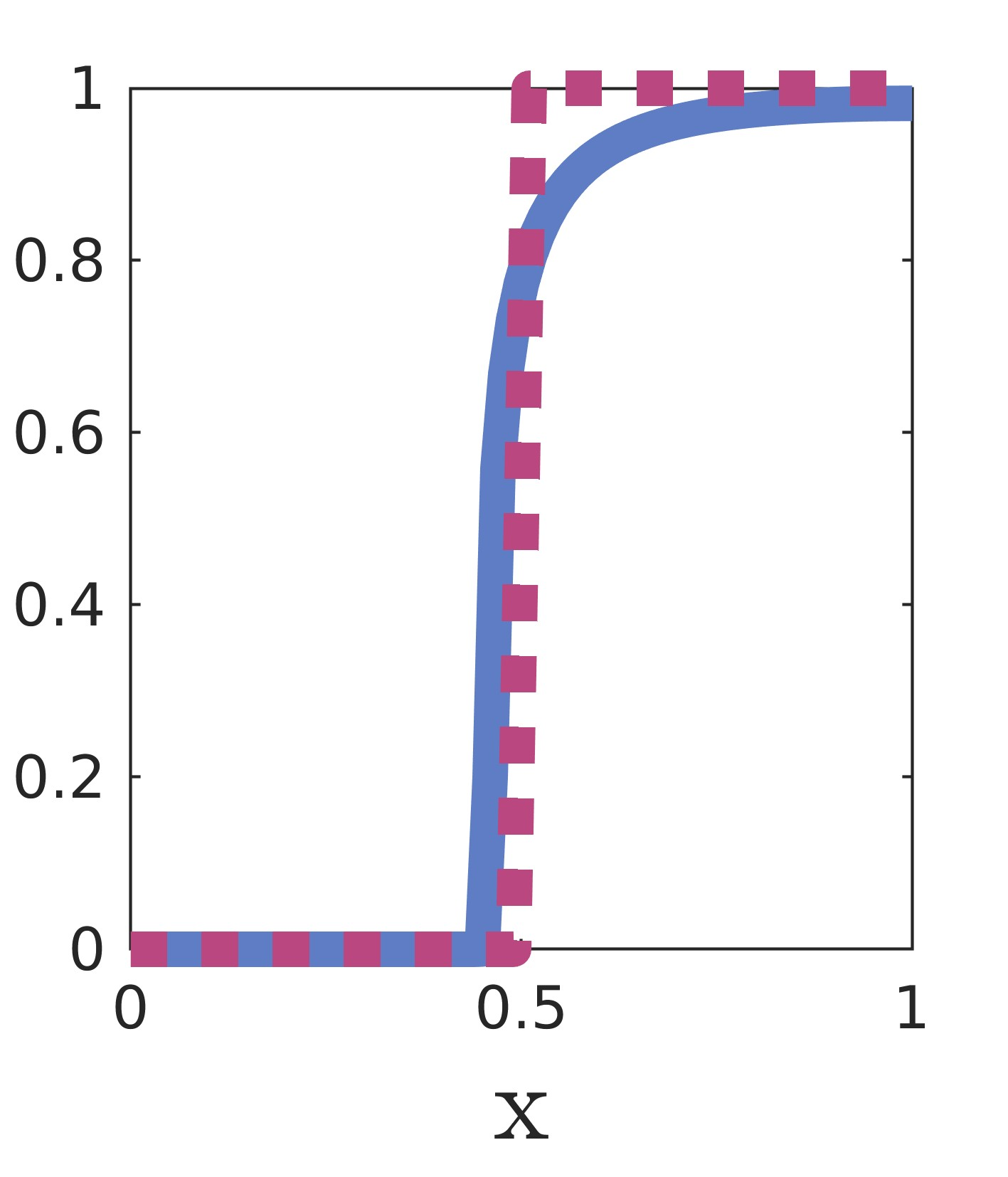}
}\\
\subfigure[$m=100$, $t=0$]{
\includegraphics[width=2.7cm,height=3cm]{Figures/uinfty1.png}
}\hspace{0.5cm}\subfigure[$m=100$, $t=5$]{
\includegraphics[width=2.7cm,height=3cm]{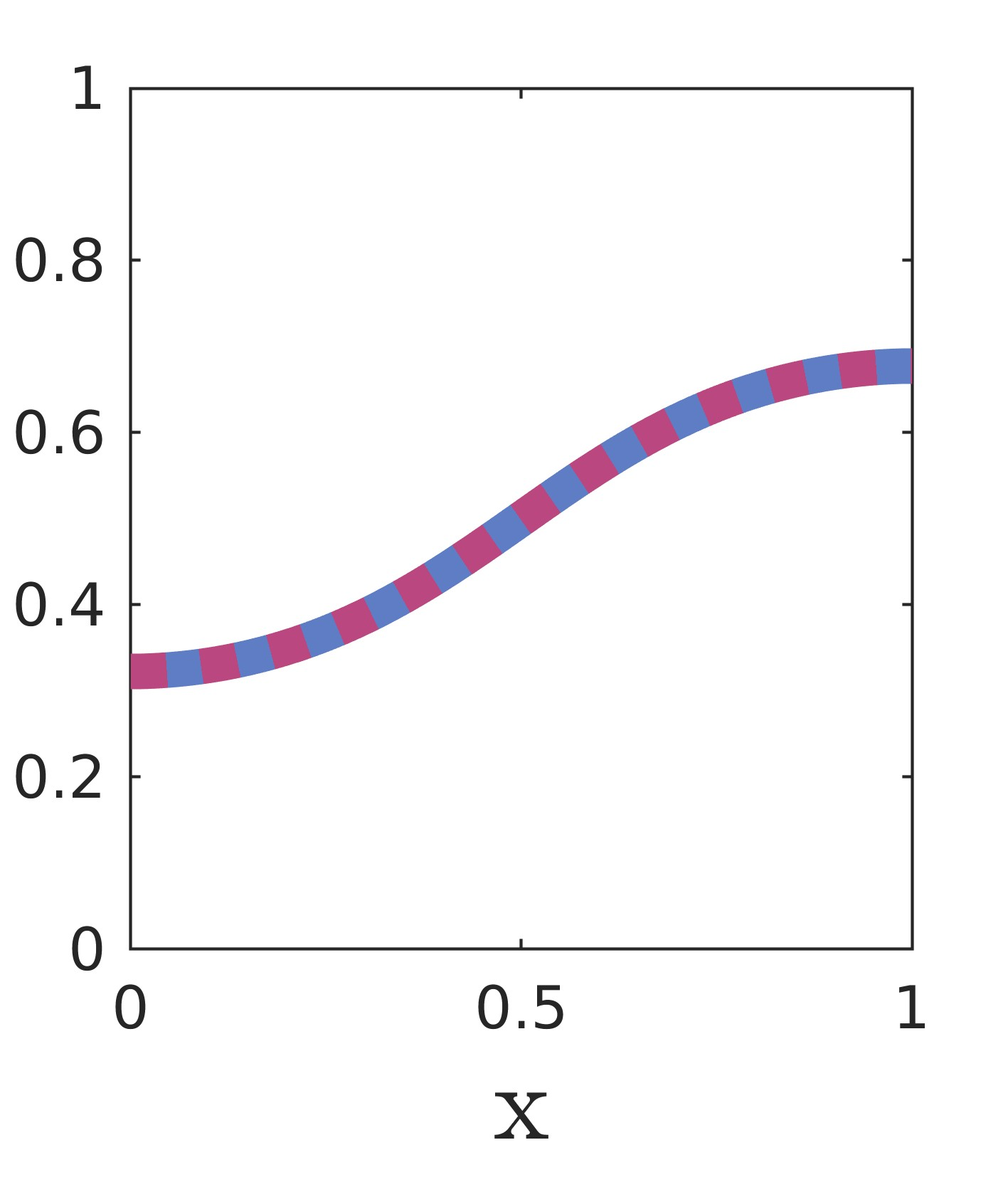}
}\hspace{0.5cm}\subfigure[$m=100$, $t=20$]{
\includegraphics[width=2.7cm,height=3cm]{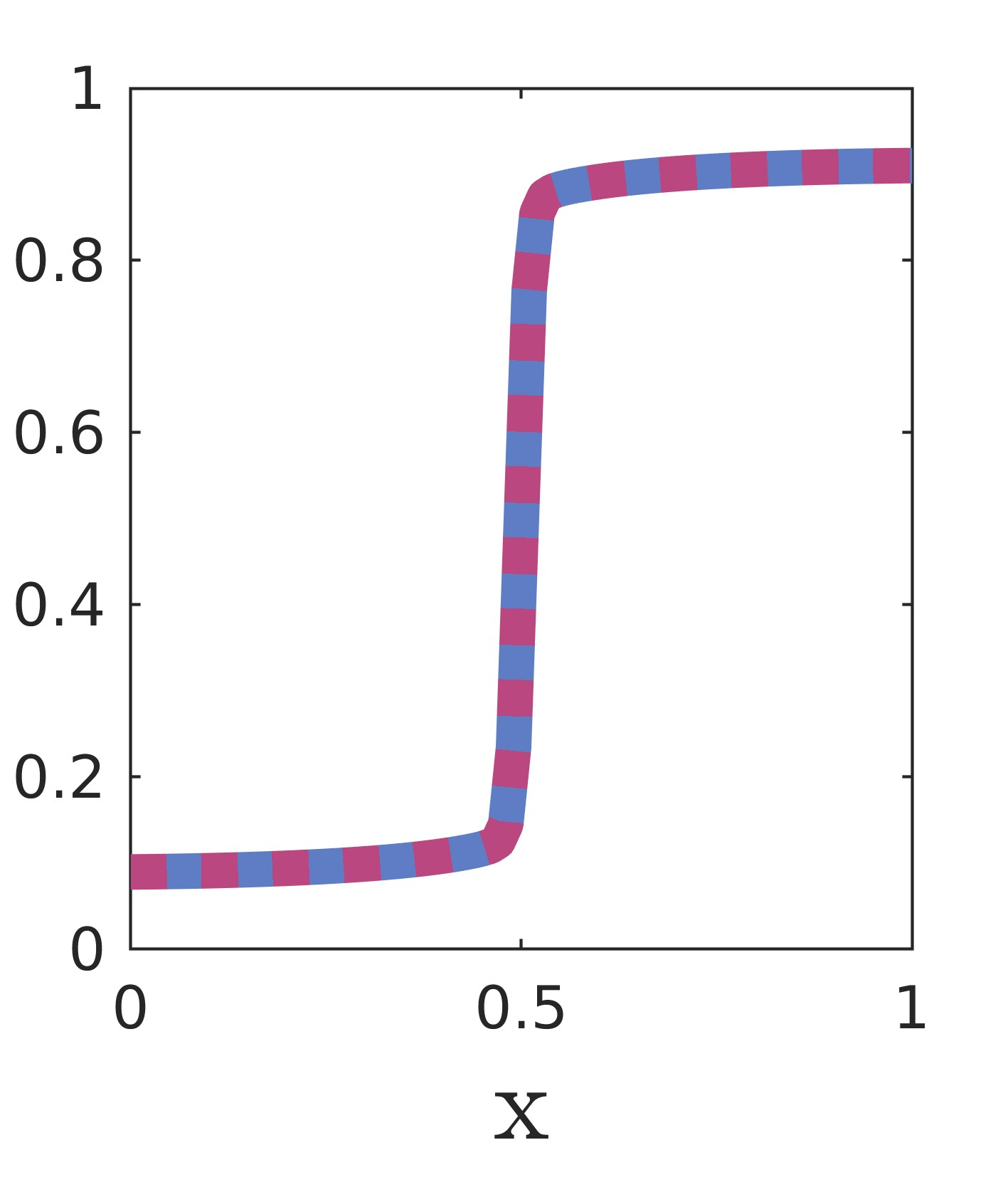}
}\hspace{0.5cm}\subfigure[$m=100$, $t=1000$]{
\includegraphics[width=2.7cm,height=3cm]{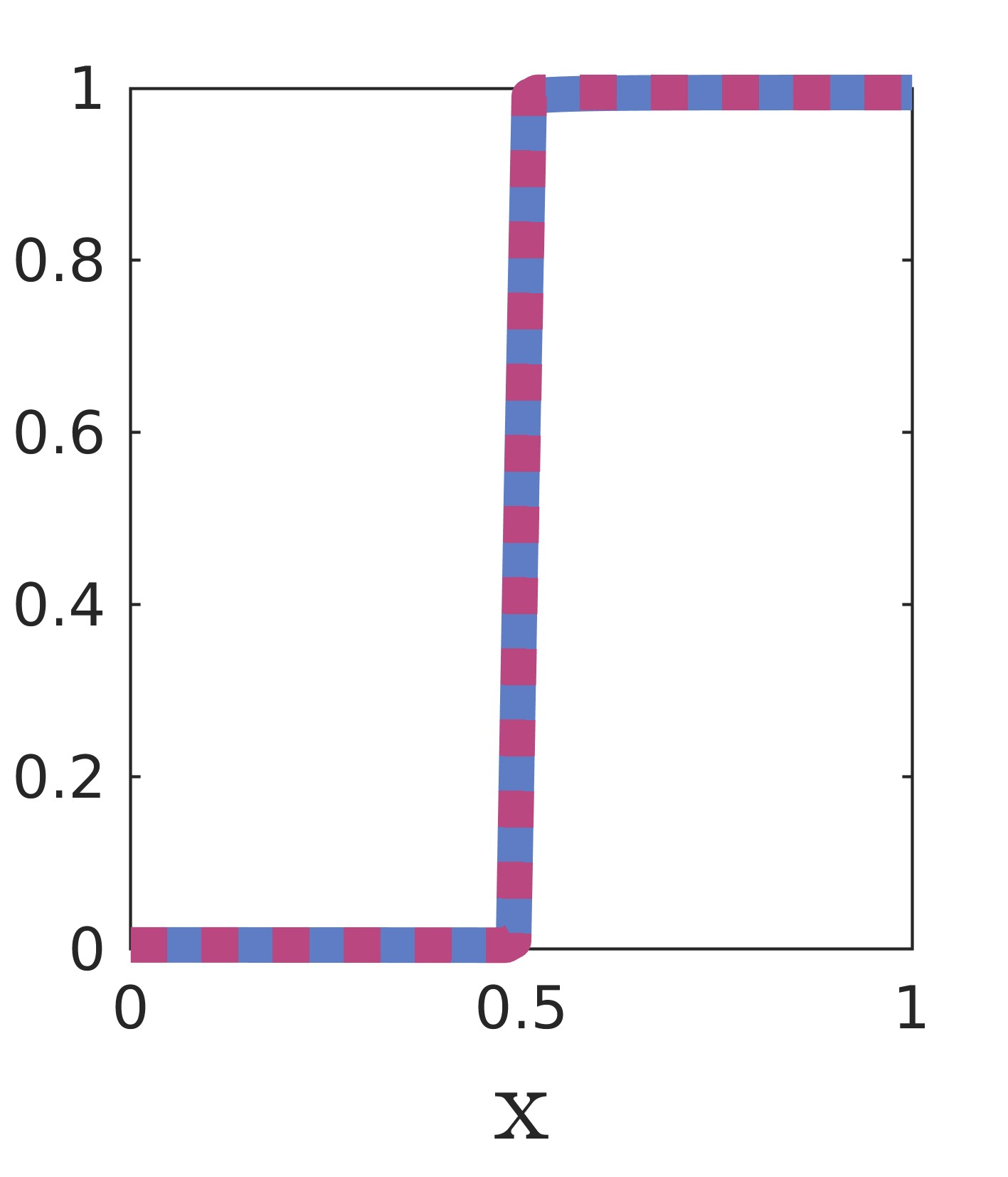}
}
\\[-8pt]
\caption{The time evolution of $u_m(t,x)$ in \eqref{deq} and $u_\infty(t,x)$ in \eqref{hss1} in one dimension with $\chi=40$, $D=1$ and $K=1$. The initial data is given by $u^0=M-0.01 \cos(\pi x)$ and $v^0=M$.}
\label{fig:label3131}
\end{figure}

\begin{figure}[ht]
\centering
\subfigure[$K=0.6$]{
\includegraphics[width=2.7cm,height=3cm]{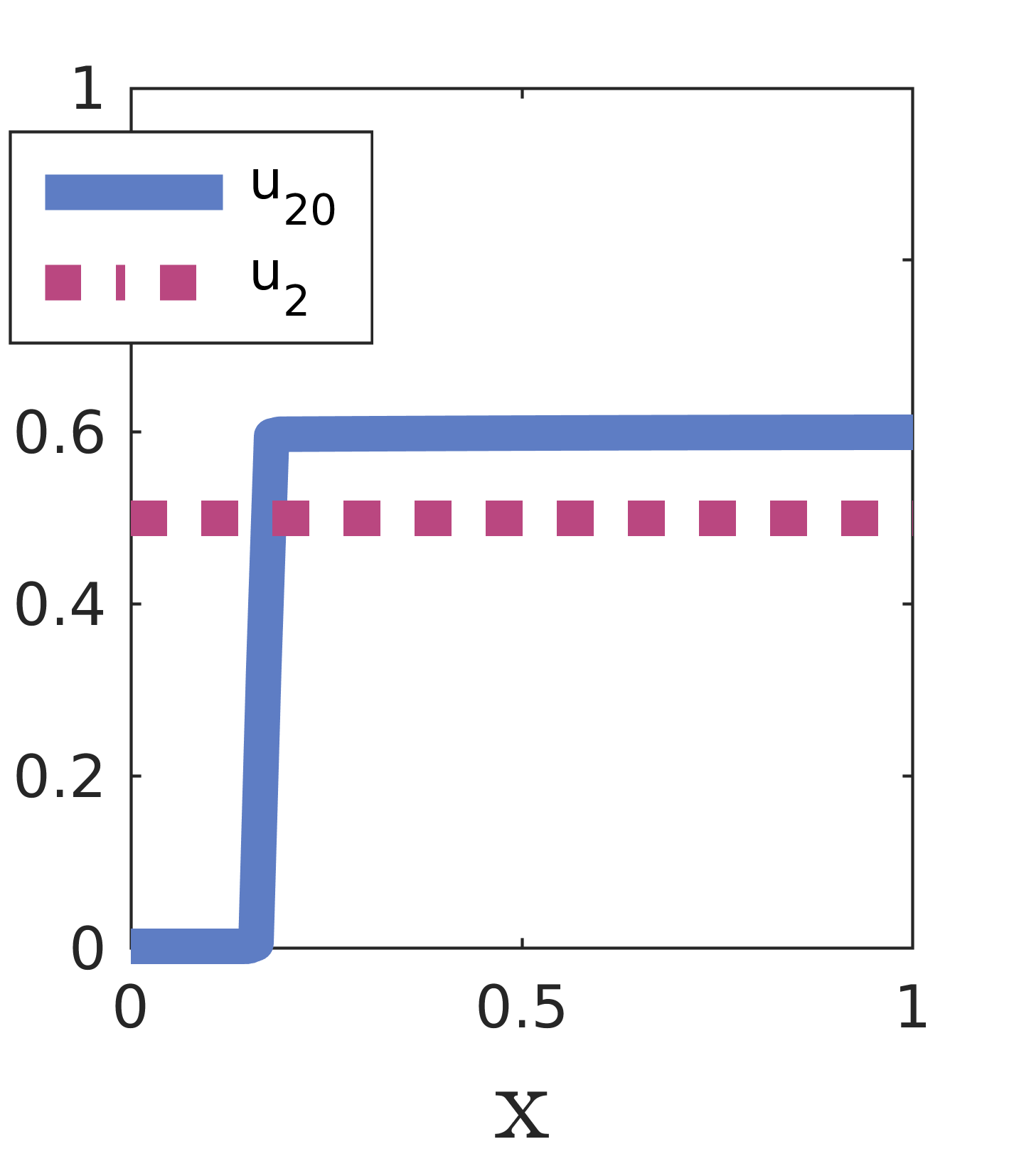}
}\hspace{0.5cm}\subfigure[$K=0.6$]{
\includegraphics[width=2.7cm,height=3cm]{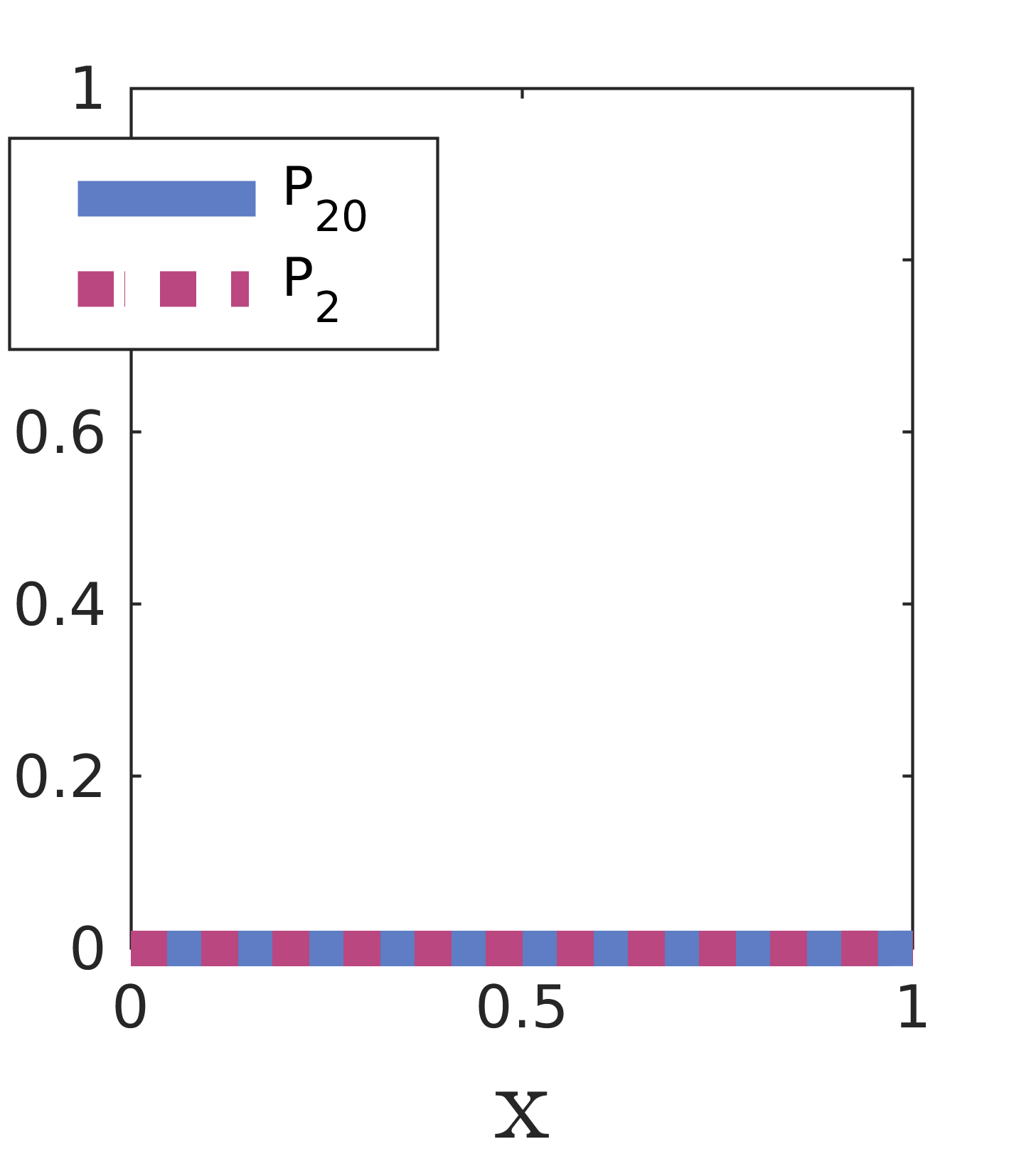}
}\hspace{0.5cm}\subfigure[$K=0.6$]{
\includegraphics[width=2.7cm,height=3cm]{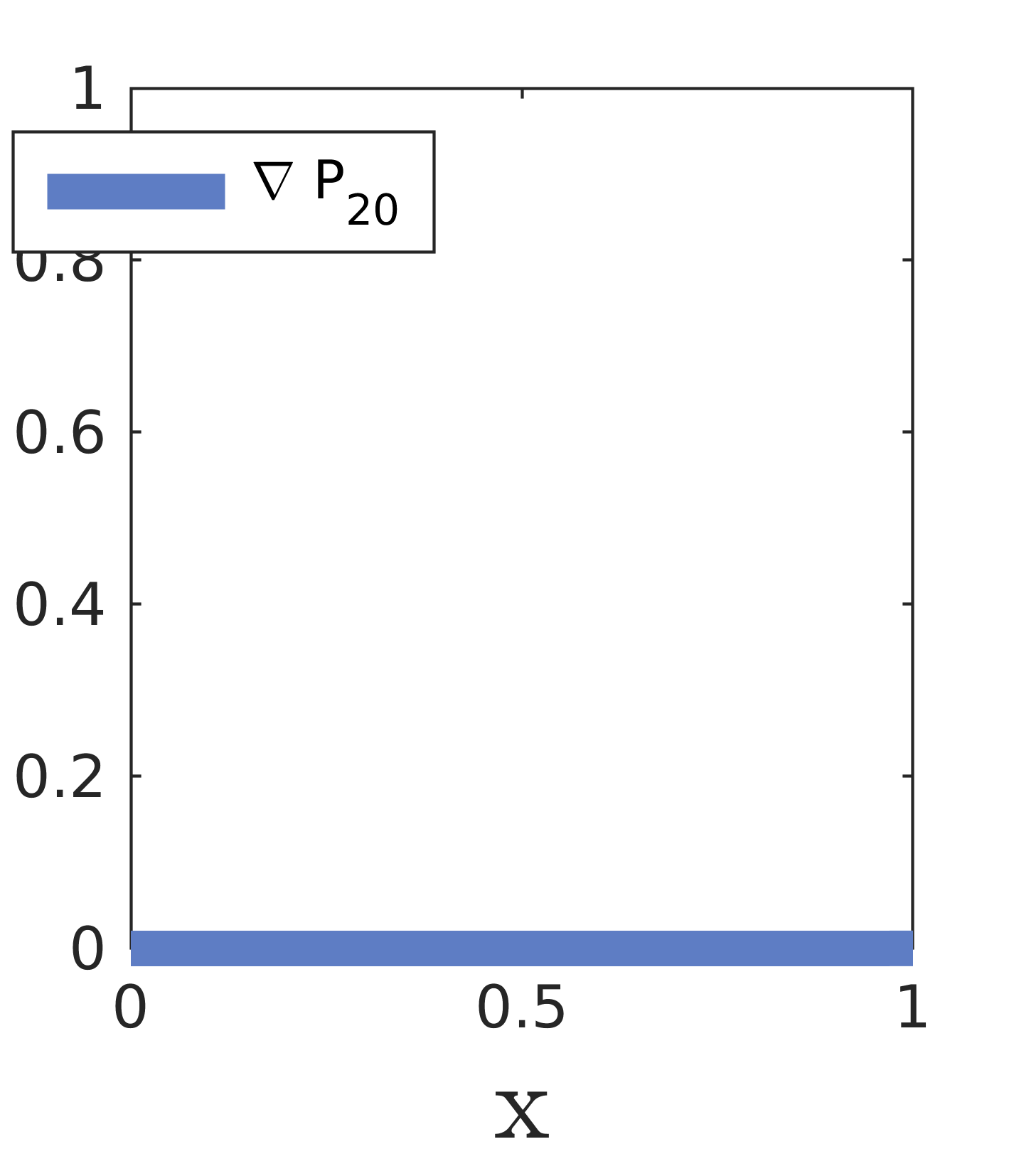}
}\hspace{0.5cm}\subfigure[$K=0.6$]{
\includegraphics[width=2.7cm,height=3cm]{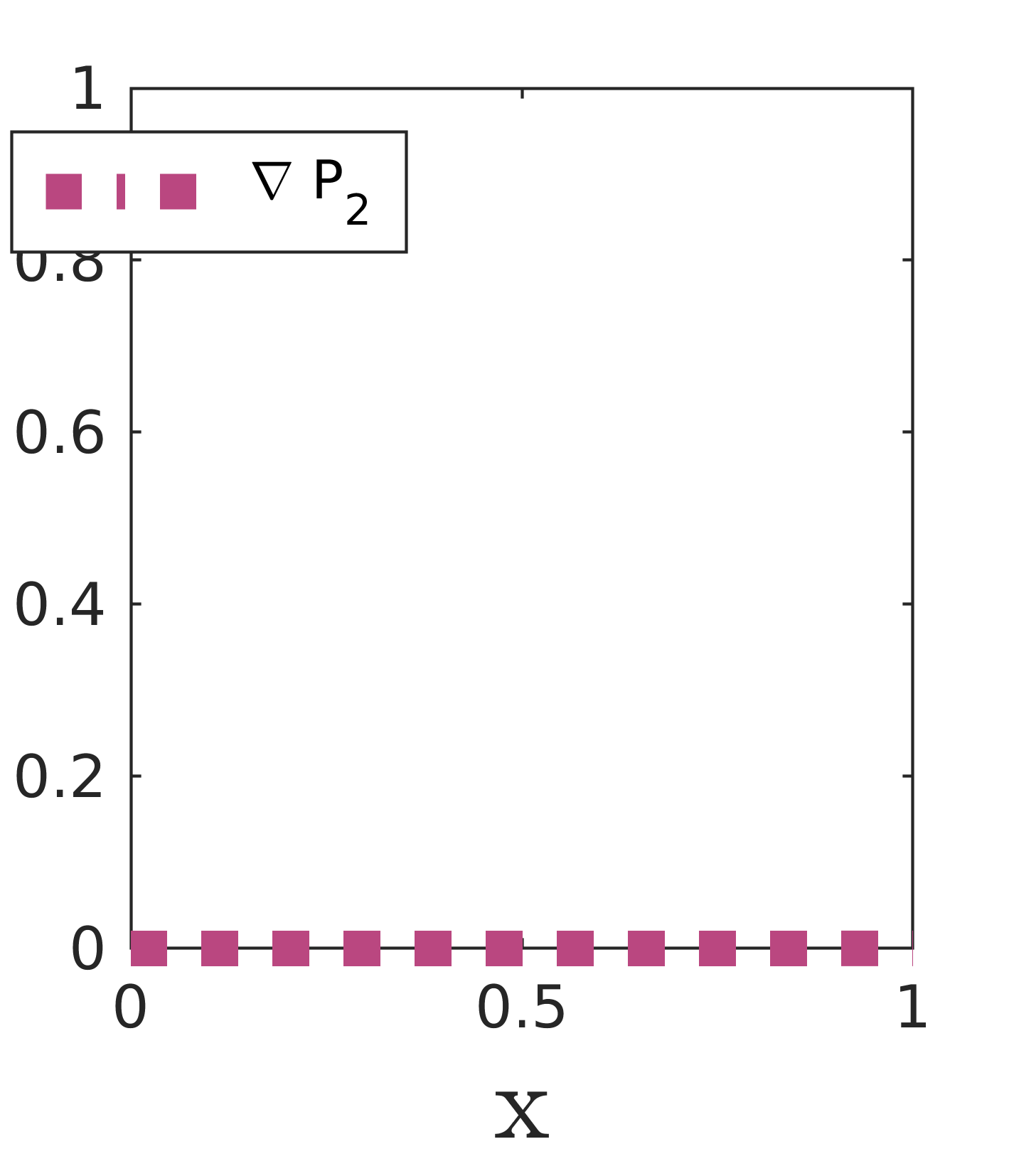}
}\\
\subfigure[$K=1$]{
\includegraphics[width=2.7cm,height=3cm]{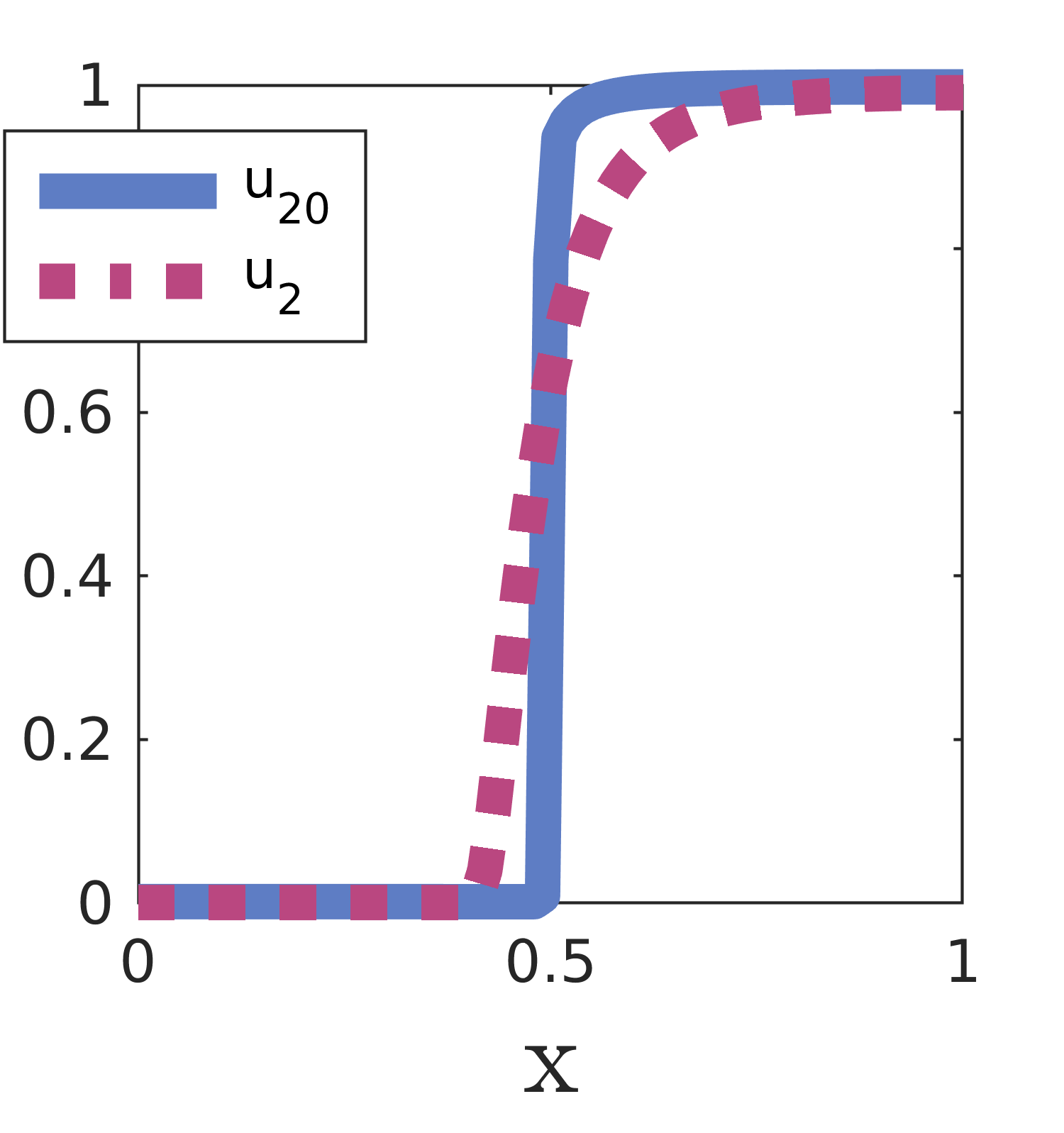}
}\hspace{0.5cm}\subfigure[$K=1$]{
\includegraphics[width=2.7cm,height=3cm]{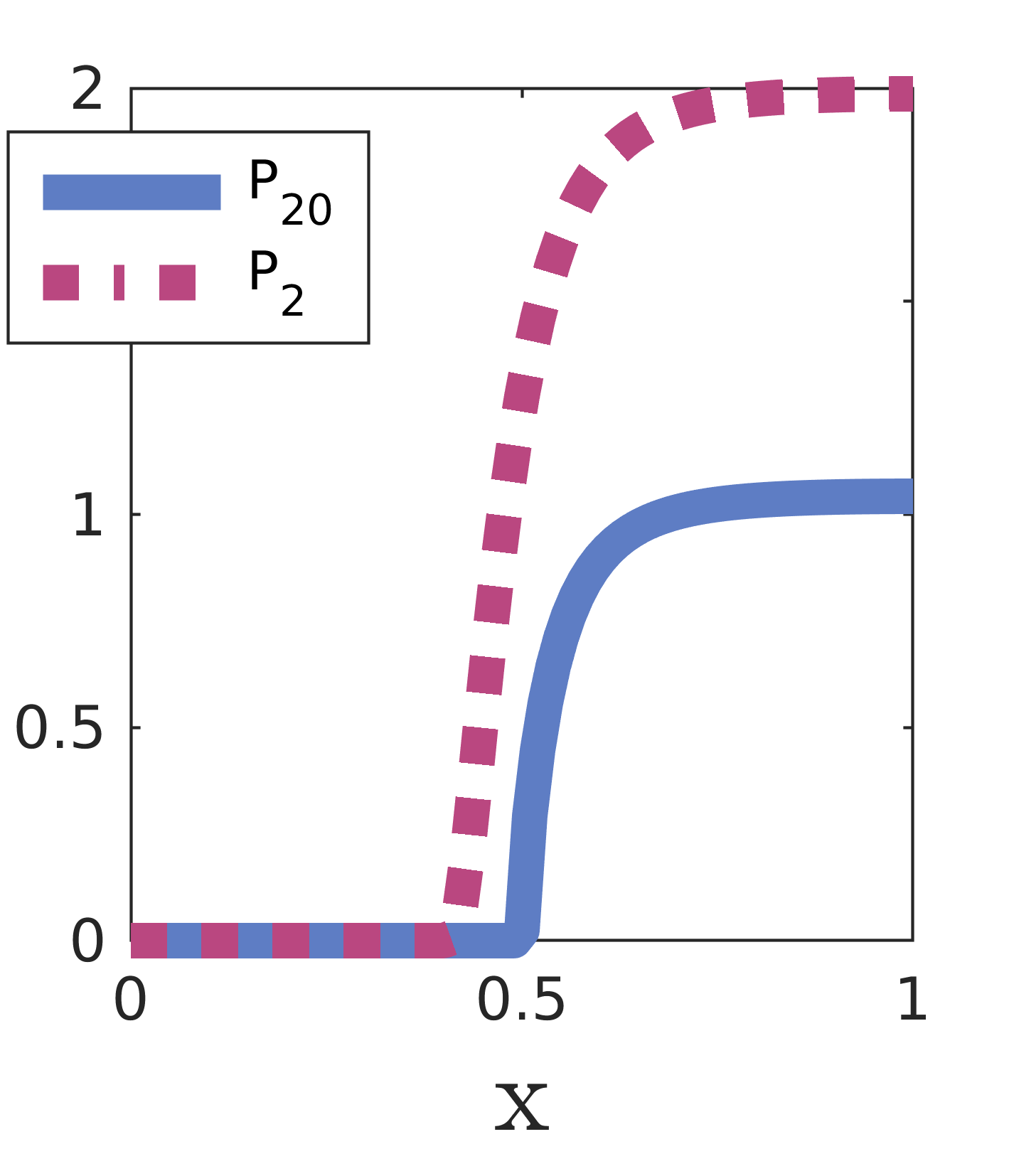}
}\hspace{0.5cm}\subfigure[$K=1$]{
\includegraphics[width=2.7cm,height=3cm]{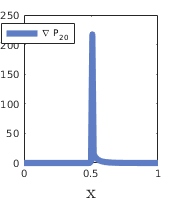}
}\hspace{0.5cm}\subfigure[$K=1$]{
\includegraphics[width=2.7cm,height=3cm]{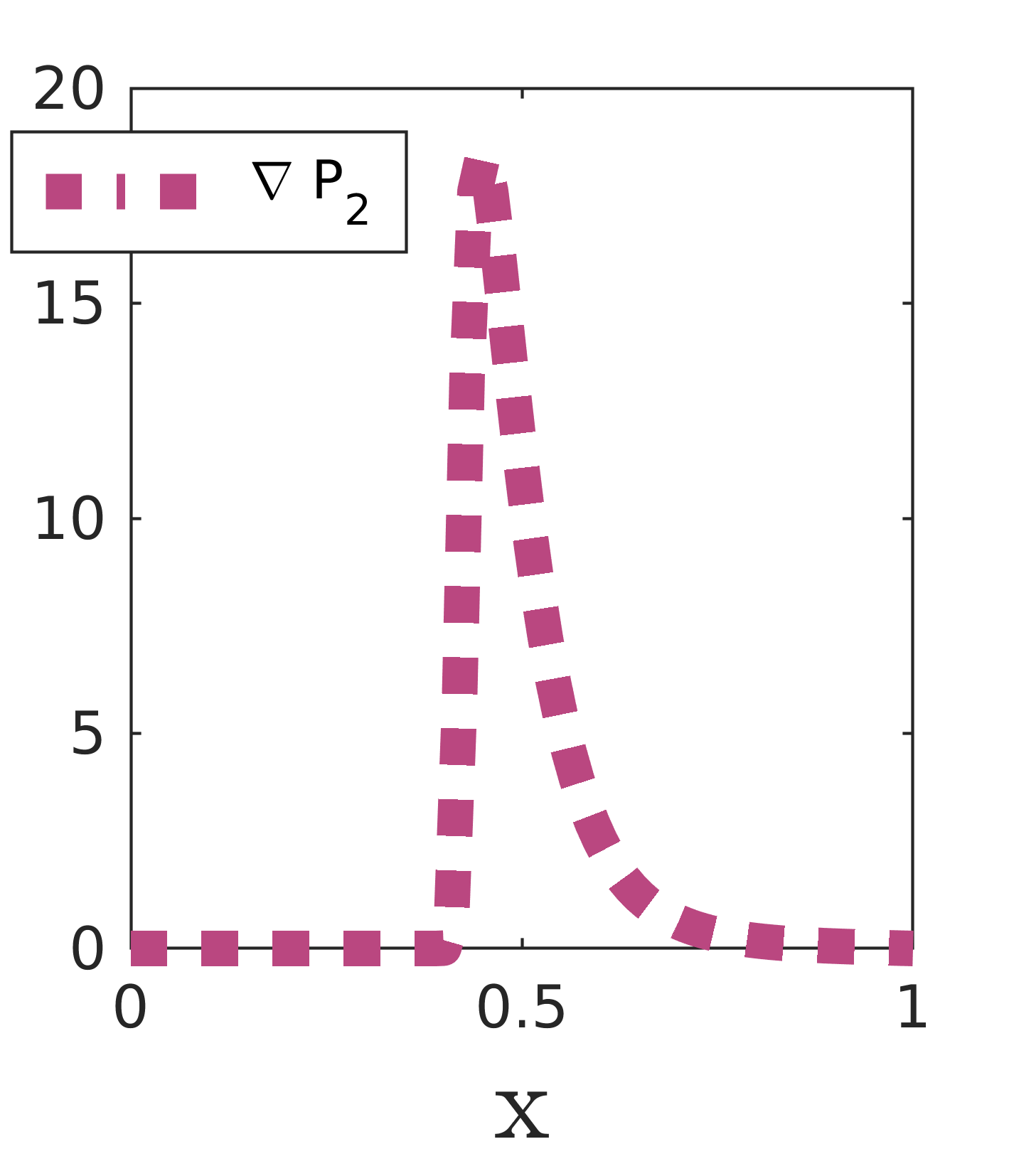}
}\\
\subfigure[$K=2$]{
\includegraphics[width=2.7cm,height=3cm]{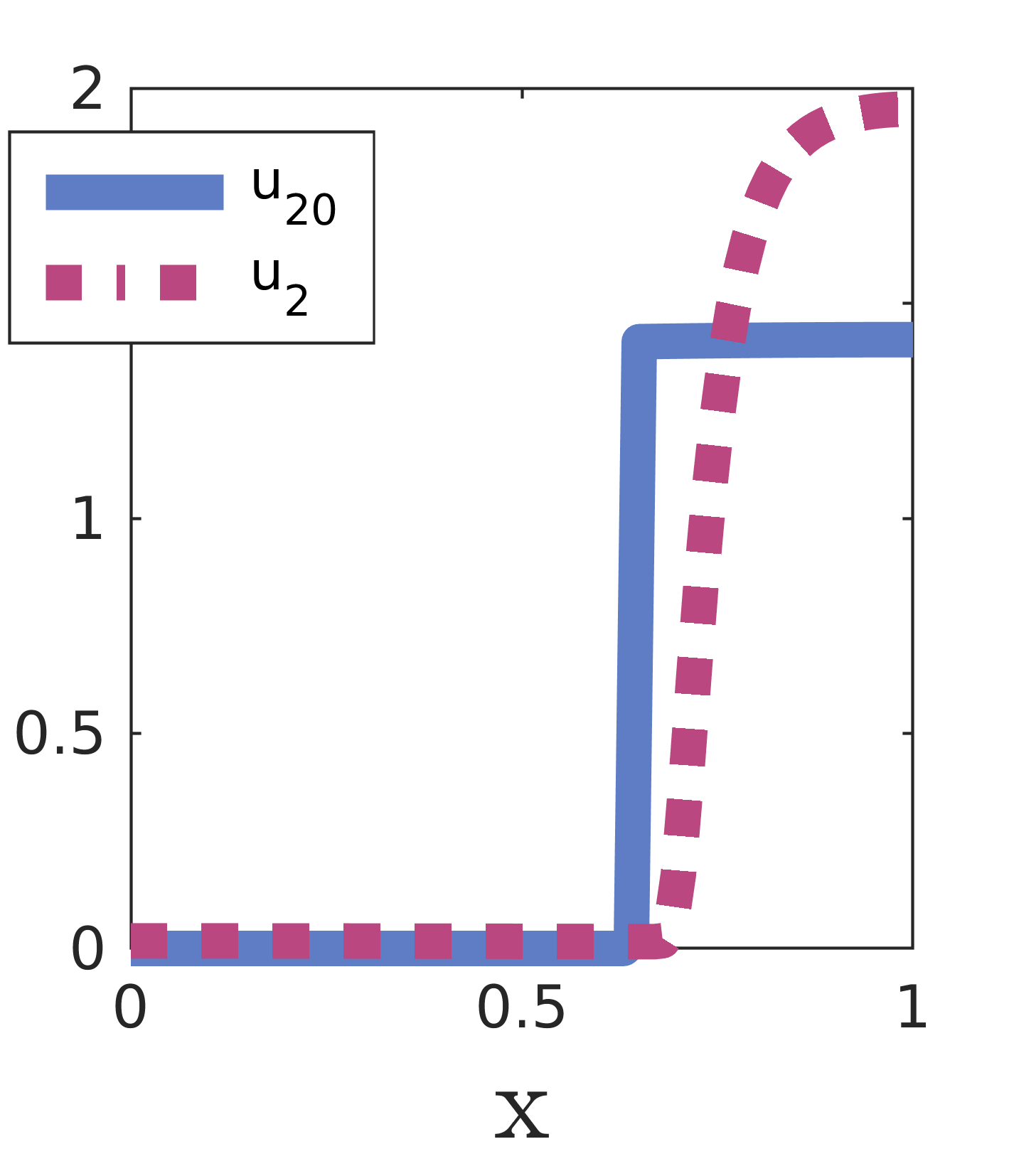}
}\hspace{0.5cm}\subfigure[$K=2$]{
\includegraphics[width=2.7cm,height=3cm]{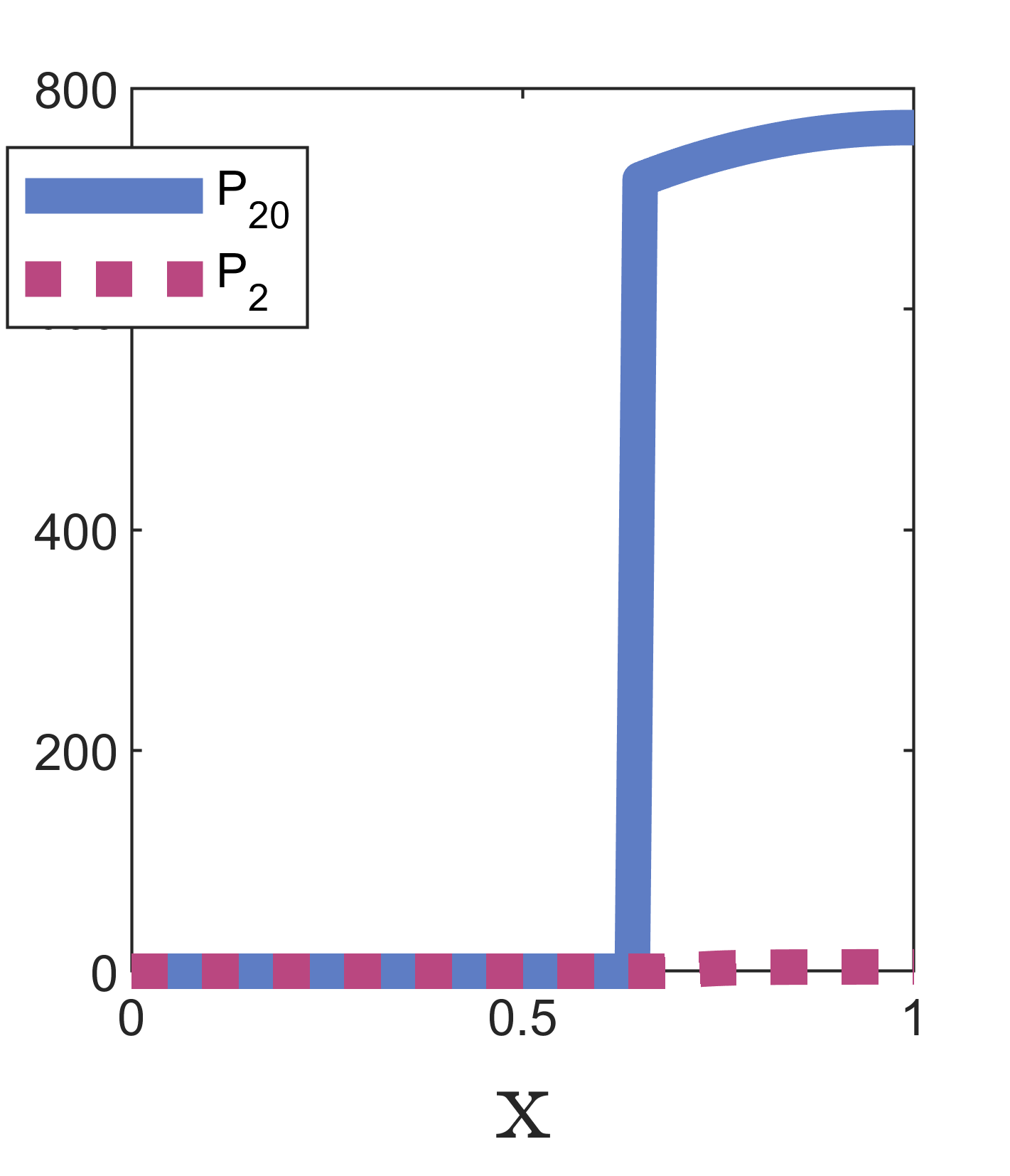}
}\hspace{0.5cm}\subfigure[$K=2$]{
\includegraphics[width=2.7cm,height=3cm]{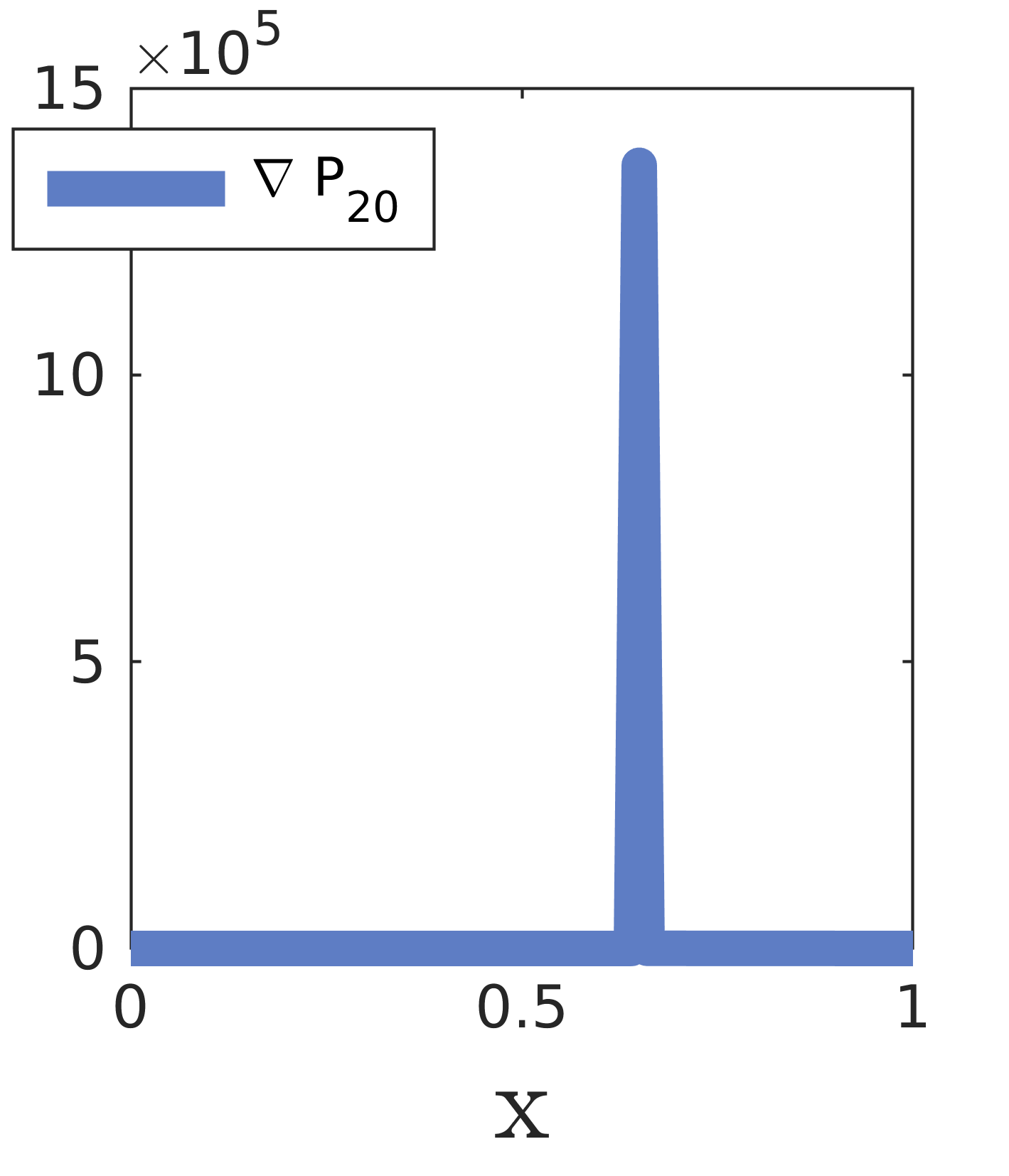}
}\hspace{0.5cm}\subfigure[$K=2$]{
\includegraphics[width=2.7cm,height=3cm]{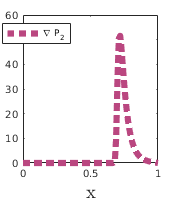}
}
\\[-8pt]
\caption{The profiles of $u_m(t,x)$  of \eqref{deq}, $P_m(t,x)$ and $\nabla P_m(t,x)$ of \eqref{hss1} in one dimension, where $u_m$ with $m=20$ and $m=2$ are denoted by $u_{20}$ and $u_2$, respectively. Other parameters are set by
$\chi=80$ and $D=1$.}
\label{fig:label31311}
\end{figure}

In Figure~\ref{fig:label3131}, we observe that for $K=1$ and $M=0.5$, as proved in Theorem \ref{hks}, the increasing solution $u_m$ of system \eqref{deq} differs from the solution $u_\infty$ of the hyperbolic system \eqref{hss2} when $m$ is small. However, as $m$ increases, $u_m$ approaches $u_\infty$. When $m$ becomes sufficiently large, the two solutions become similar at any time, illustrating that system \eqref{deq} converges to system \eqref{hss2} as $m\to \infty$.

In Figure~\ref{fig:label31311}, we present the profiles of the pressure $P_m$, $\nabla P_m$, and the corresponding density $u_m$. For $K=0.6$, we observe that both $P_m$ and $\nabla P_m$ vanish when $m$ is large, which illustrates \eqref{pressurev} in Theorem~\ref{hks}. For $K=1$, compared to $m=2$, the interval where $\nabla P_m>0$ decreases as $m$ increases, indicating the validity of \eqref{pressurev}. For $K=2$, $u_{20}$ has an upper bound smaller than that of $u_2$. We can observe that $\nabla P_{20}=0$ almost everywhere illustrates \eqref{pressurev}.

\section{Conclusion and perspectives}

In the present paper, we studied the incompressible limit of the system \eqref{deq}. The novel finding of our work is that the different choices of $K$ lead to different limiting systems, in other words, to different phenomenas. While the supercritical case $K>1$ forms a Hele-Shaw free boundary equation. In the subcritical case $K\leq 1$, the stiff pressure effect, the main feature of the Hele-Shaw problem, vanishes and the incompressible limit solves a hyperbolic Keller-Segel system \cite{bp2009}. To this end, the strong limits of both pressure and density are required. Therefore, we proved the strong limit of the pressure for all $K>0$ by making  use of the porous medium type structure, and verified the strong convergence of the density via the kinetic formulation method for $K\leq1$.  An open problem we left is to verify the strong limit of the density for the case $K>1$.

 In addition, there are some open problems for the Keller-Segel system~\eqref{deq} with the volume filling effect in $\mathbb{R}^n$. First of all, since it can be observed from the numerical simulation that a sharp interface exists on the density and pressure when $m$ is large, the regularity of $u_m$ and $P_m$ is worth to be further understood; cf.~\cite{MPQ2017}. It is also challenging to prove the existence of the non-monotonic solution, which is formed by the increasing solutions via reflection, where it is positive in some regions surrounded by two regions
of vacuum. Furthermore, we would like to know if the non-monotonic solution has some wavelike properties; see~\cite{DQZ2020}. Moreover, one can study its existence and the radially symmetric property of the stationary state inspired by \cite{Carrillo2019nonlinear,CHMV2018}. The other challenging problem is to study its large time asymptotic behaviour in which it is difficult to obtain the time uniform moment estimates of the density; see \cite{Carrillo2019nonlinear}.

\bigbreak
\noindent\textbf{Acknowledgments}  Qingyou He is supported by the 
ANR project ChaMaNe (ANR-19-CE40-0024). The authors would like to thank Prof.~Beno{\^i}t Perthame for the fruitful discussions throughout the preparation of this paper.
\appendix
\section{Decomposition of limiting pressure}
From Proposition~\ref{complementarityr}, we know that the stiff pressure effect $\Delta P_\infty$ vanishes when $K\leq1$. For $K>1$,
inspired by \cite{MRS2010,KM2023,GKM2022}, we study $\Delta P_\infty$ by decomposition, and conclude that $\Delta P_\infty$ is a non-positive bounded measure  on the interior support of $P_\infty$. 

\vspace{2mm}

For the Lebesgue measurable function $P_\infty \in L^2(0,T;H^1(\Omega))$ and any $T>0$, 
we can define the measurable  set 
\begin{equation}\label{support}
\mathcal{P}(t)=:\{x\in \Omega : P_\infty(x,t)>0\},\quad \text{ a.e. }t>0.
\end{equation}
For point-wise characterization of the density $u_\infty$, we define the support of the measure $1-u_\infty$ by 
\[\text{Supp}(1-u_\infty(t)):=\{x_0\in \Omega: \int_{B_r(x_0)\cap\Omega}(1-u_\infty)(\cdot,t)dx>0\text{ for all }r>0\},\]
which is closed in $\Omega$ by its definition. The corresponding complement, open in $\Omega$, is given by 
\begin{equation}\label{interior}
\mathcal{O}(t):=\{x\in \Omega: \int_{B_r(x)\cap\Omega}(1-u_\infty)(\cdot,t)dx=0\text{ for some }r>0\}.
\end{equation}

\vspace{3mm}

We show that the complementarity relation~\eqref{CR} holds for the more general test functions. 
\begin{proposition}\label{ecr}
For any $T>0$ and $\phi_{t}\in H^1(\Omega)$ with $\phi_{t}(x)(1-u_\infty(x,t))=0$,\text{ a.e }$x\in\Omega$, $t\in[0,T]$, $P_\infty(\cdot,t)$ solves  
\begin{equation}\label{k1}\int_{\Omega}-\nabla P_\infty(x,t)\cdot\nabla \phi_t(x)+(K-1)(1-c_\infty(x,t))\phi_t(x)dx=0,\end{equation}
up to a set of measure zero in $t>0$. Furthermore, we have 
\begin{equation}\label{k2}
   \Delta P_\infty+(K-1)(1-c_\infty)=0, \quad\text{in }\mathcal{D}'(\mathcal{O}(t)).
\end{equation}
\end{proposition}

\vspace{3mm}

We decompose the stiff pressure effect ($\Delta P_\infty$) as in \cite{KM2023}. Define 
\[\mu_t:=\Delta P_\infty(x,t)+(K-1)(1-c_\infty(x,t))\textbf{1}_{\mathcal{P}(t)}\quad \text{in }\mathcal{D}'(\Omega).\]
Then we show that $\mu_t$ is a non-negative measure. 
\begin{proposition}\label{dp}
For $K>1$ and almost everywhere $t>0$, up to a zero measure in time $t$, it follows
\begin{equation}\label{suppnu}
supp(\mu_t)\subset \partial \mathcal{P}(t)\backslash\mathcal{O}(t),\quad \mu_t\geq 0\text{ on }\Omega,\quad   \text{and }\mu_t(\textbf{1}_{\Omega})< C.
\end{equation}
\end{proposition}
\noindent\underline{\textbf{Proof of Proposition~\ref{ecr}}}
Under the assumptions that $\phi_{t}\in H^1(\Omega)$ and $\phi_{t}(x)(1-u_\infty(x,t))=0$ a.e. in $x\in\Omega$, similar to the proof of Lemma~\ref{cr}, we conclude
\begin{equation}\label{r232}
(1-u_\infty(x,t))\nabla \phi_{t}=0,\quad(1-\overline{u_\infty^2}(x,t)) \phi_{t}=0,\quad (1-\overline{u_\infty^2}(x,t))\nabla \phi_{t}=0,\quad \text{ a.e. } x\in\Omega.  
\end{equation}

For any $T>0$, let 
$u_\infty,\overline{u^2_\infty}\in L^\infty(Q_T)$, $P_\infty\in L^2(0,T;H^1(\Omega))$
and $c_\infty\in L^p(0,T;W^{1,p}(\Omega))$
for all $p\in[1,\infty)$ be the weak limits of $u_m,u_m^2,P_m,c_m$, respectively. Then, we pass to the limit in \eqref{deq}$_1$ as $m\to\infty$. By \eqref{umss}, we get
\begin{equation*}
    \partial_tu_\infty=\Delta P_\infty-\nabla\cdot((Ku_\infty-\overline{u_\infty^2})\nabla c_\infty),\quad \text{in  }\mathcal{D}'(Q_T).
\end{equation*}
For $t<T$, consider $h>0$ such that $0< t+h\le T$. Using \eqref{r232}, it holds
\[\begin{aligned}\int_{t}^{t+h}\hskip-4pt\int_\Omega&-\nabla\phi_t\cdot\nabla P_\infty(x,s)+(Ku_\infty-\overline{u_\infty^2})\nabla c_\infty\cdot\nabla\phi_{t}(x) dxds\\
=&\int_{\Omega}[u_\infty(x,t)-u_\infty(x,t+h)]\phi_{t}(x) dx=\int_{\Omega}[1-u_\infty(x,t+h)]\phi_{t}(x) dx\geq0,\end{aligned}\]
and it further follows
\[\frac{1}{h}\int_{t}^{t+h}\hskip-6pt\int_\Omega-\nabla\phi_{t}\cdot\nabla P_\infty(x,s)+(Ku_\infty-\overline{u_\infty^2})\nabla c_\infty(x,s)\cdot\nabla\phi_{t}(x) dxds
\geq0.\]
Let $h\to 0^+$ and use Lebesgue's differentiation theorem, up to a set of measure zero in $t>0$, we obtain 
\[\int_\Omega-\nabla\phi_{t}(x)\cdot\nabla P_\infty(x,t)+(Ku_\infty(x,t)-\overline{u_\infty^2}(x,t))\nabla c_\infty(x,t)\cdot\nabla\phi_{t}(x) dx
\geq0.\]
By \eqref{r232}, we deduce
\[\int_\Omega(K-1)\nabla c_\infty(x,t)\cdot\nabla\phi_{t}(x) dx
\geq \int_\Omega\nabla\phi_{t}(x)\cdot\nabla P_\infty(x,t) dx,\]
which derives 
\[\int_\Omega(K-1)(1-c_\infty(x,t))\phi_{t}(x) dx
\geq \int_\Omega\nabla\phi_{t}(x)\cdot\nabla P_\infty(x,t) dx,\quad \text{up to a zero measure set in }t>0.\]

Conversely, considering $h\to0^-$, then we get the corresponding result 
\[\int_\Omega(K-1)(1-c_\infty(x,t))\phi_{t}(x) dx
\leq \int_\Omega\nabla\phi_{t}(x)\cdot\nabla P_\infty(x,t) dx,\quad\text{up to a zero measure set in }t>0.\]
Hence, we deduce \eqref{k1}, and further derives \eqref{k2}. 

\vspace{4mm}

\noindent\underline{\textbf{Proof of Proposition~\ref{dp}}} \emph{Step 1.}
We first prove the support property of \eqref{suppnu}. By the definition of $\mu_t$, for any $\varphi\in \mathcal{C}^1(\overline{\Omega})$, it holds
\begin{equation}\label{use}
\mu_t(\varphi)=-\int_{\Omega}\nabla P_\infty\cdot\nabla \varphi dx+\int_{\Omega}(K-1)(1-c_\infty)\textbf{1}_{\mathcal{P}(t)}\varphi dx,\quad \text{a.e. }t>0.\end{equation}
On the one hand, for $\varphi$ is supported in $\mathcal{O}(t)$ (defined by (\eqref{interior}), by means of  \eqref{CR}, we have 
\[\mu_t(\varphi)=0, \text{ a.e. }t> 0, \]
which implies
\[
\text{supp}(\mu_t)\cap \mathcal{O}(t)=\emptyset.
\]
On the other hand, when $\varphi$ is supported in $\{x\in\Omega: P_\infty(x,t)=0\}$, due to $P_\infty(t)\in H^1(\Omega)$, which implies $\nabla P_\infty=0$ a.e. in $\{P_{\infty}(\cdot,t)=0\}$,  \eqref{use} yields 
\[\mu_t(\varphi)=0,\text{ a.e. }t>0,
\]
which indicates
\[
\text{supp}(\mu_t)\cap\text{Int}(\{P_\infty(\cdot,t)=0\})=\emptyset.
\]

In addition, notice that $\text{Int}(\mathcal{P}(t))\subset \mathcal{O}(t)$. Indeed,
for $P_\infty(\cdot,t)>0$ in $\text{B}_\delta(x_0)$, from \eqref{sr1}, we have
\[
1-u_\infty(t)=0,\quad\text{a.e. } \text{in}\ \text{B}_\delta(x_0)\quad \Longrightarrow\quad  \int_{\text{B}_\delta(x_0)}1-u_\infty(t) dx=0,
\]
this follows that $x_0\in\mathcal{O}(t)$. Therefore, we have the first result of \eqref{suppnu}.
\\

\noindent\emph{Step 2.} We show the non-negative property of the measure $\mu_t$. We define the cut-off function 
\begin{equation*}
H_\delta(s)=\begin{cases}
\frac{s}{\delta},\quad&0\leq s \leq \delta,\\
1,\quad &s>\delta. 
\end{cases}
\end{equation*}
For any test function $\varphi\in \mathcal{D}(\Omega)$ with $0\leq \varphi\leq 1$, we use the result~\eqref{k1} via $\phi_t(x):=\varphi(x)H_\delta(P_\infty(x,t))$ and obtain 
\begin{equation*}
\begin{aligned}
\mu_t(\varphi)
=&\underbrace{\int_{\Omega}-\nabla P_\infty\cdot\nabla (\varphi H_\delta(P_\infty)+(K-1)(1-c_\infty)\textbf{1}_{\mathcal{P}(t)}\varphi H_\delta(P_\infty)dx}_{=0}\\
&+\langle\Delta P_\infty,\varphi(1-H_\delta(P_\infty))\rangle+\underbrace{\int_{\Omega}(K-1)(1-c_\infty)\textbf{1}_{\mathcal{P}(t)}\varphi[1-H_\delta(P_\infty)]dx}_{\geq 0}\\
\geq &\langle\Delta P_\infty,\varphi(1-H_\delta(P_\infty))\rangle.
\end{aligned}
\end{equation*}
We note 
\begin{equation*}
\begin{aligned}
\langle \Delta P_\infty,\varphi(1-H_\delta(P_\infty))\rangle=&\int_{\Omega}\nabla P_\infty\cdot \nabla\varphi [H_\delta(P_\infty)-1]dx+\int_{\Omega}|\nabla P_\infty|^2 H_\delta'(P_\infty)\varphi dx\\
\geq &\int_{\Omega}\nabla P_\infty\cdot \nabla\varphi [H_\delta(P_\infty)-1]dx.
\end{aligned}
\end{equation*}
Since $H_\delta(P_\infty)\to \textbf{1}_{\Omega}-\textbf{1}_{\{P_\infty=0\}}$ a.e. $x\in \Omega$ as $\delta\to 0^+$, using the dominated convergence theorem, it holds 
\begin{equation*}
\int_{\Omega}\nabla P_\infty\cdot \nabla\varphi [H_\delta(P_\infty)-1]dx\to 0,\text{ as }\delta\to 0^+.
\end{equation*}
Which, combined with the arbitrariness of $\delta\geq0$, implies
\begin{equation*}
\mu_t(\varphi)\geq 0.
\end{equation*}
Finally, we have
\[\mu_t(\textbf{1}_{\Omega})=\int_{\Omega}(K-1)(1-c_\infty)\textbf{1}_{\mathcal{P}(t)} dx\leq C,\quad\text{a.e. }t>0.\]

\hfill

\section{Preliminary lemmas}
We present some useful lemmas.
\begin{lemma}[Aubin-Lions-Simon,\cite{simon1986}]\label{al}Let $X_0$, $X$, and $X_1$ be three Banach Spaces with $X_0\subset X\subset X_1$. Suppose that $X_0$ is compactly embedded in $X$ and that $X$ is continuously embedded in $X_1$. Let 
\[W_{p,q}(0,T):=\{u\in L^p(0,T:X_0): \partial_t u\in L^q(0,T;X_1)\}\]
 with $1\leq p,q\leq \infty$, then the following hold:
\begin{itemize}
    \item[(\romannumeral1)] If $p<\infty$, then the embedding of $W_{p,q}(0,T)$ into $L^p(0,T;X)$ is compact. 
    \item[(\romannumeral2)]If $p=\infty$ and $q>1$, then the embedding of $W_{p,q}(0,T)$ into $\mathcal{C}(0,T;X)$ is compact.
\end{itemize}
\end{lemma}
\begin{lemma}[cf.~\cite{gn1983}]\label{ei}Let $\Omega$ be a bounded domain in $\mathbb{R}^n$ with Lipschitz boundary and $p^*:=\frac{np}{n-p}$ with $1\leq p<n$. Then there is a positive constant $c$, depending on $n$, such that 
\[\|u-u_S\|_{L^{p^*}(\Omega)}\leq \frac{c d^{n+1-\frac{n}{p}}}{|S|^{\frac{1}{p}}}\|\nabla u\|_{L^p(\Omega)},\quad \forall\ u\in W^{1,p}(\Omega),\]
where $S$ is any measurable subset of $\Omega$ with $|S|>0$, $u_S=\frac{1}{|S|}\int_{S}udx$, and $d$ is the diameter of $\Omega$.
\end{lemma}

\bibliographystyle{abbrv}
\parskip=0pt
\small
\bibliography{Paper}

\begin{thebibliography}{10}

\bibitem{bertozzi2010existence}
A.~L. Bertozzi and D.~Slep\v{c}ev.
\newblock Existence and uniqueness of solutions to an aggregation equation with
  degenerate diffusion.
\newblock {\em Commun. Pure Appl. Anal.}, 9(6):1617--1637, 2010.

\bibitem{bubba2020discrete}
F.~Bubba, T.~Lorenzi, and F.~R. Macfarlane.
\newblock From a discrete model of chemotaxis with volume-filling to a
  generalized {P}atlak-{K}eller-{S}egel model.
\newblock {\em Proc. A.}, 476(2237):20190871, 19, 2020.

\bibitem{BPPS2020}
F.~Bubba, B.~Perthame, C.~Pouchol, and M.~Schmidtchen.
\newblock Hele-{S}haw limit for a system of two reaction-(cross-)diffusion
  equations for living tissues.
\newblock {\em Arch. Ration. Mech. Anal.}, 236(2):735--766, 2020.

\bibitem{CarrilloCalvez2006}
V.~Calvez and J.~A. Carrillo.
\newblock Volume effects in the {K}eller-{S}egel model: energy estimates
  preventing blow-up.
\newblock {\em J. Math. Pures Appl. (9)}, 86(2):155--175, 2006.

\bibitem{Carrillo2019nonlinear}
J.~A. Carrillo, S.~Hittmeir, B.~Volzone, and Y.~Yao.
\newblock Nonlinear aggregation-diffusion equations: radial symmetry and long
  time asymptotics.
\newblock {\em Invent. Math.}, 218(3):889--977, 2019.

\bibitem{CHMV2018}
J.~A. Carrillo, F.~Hoffmann, E.~Mainini, and B.~Volzone.
\newblock Ground states in the diffusion-dominated regime.
\newblock {\em Calc. Var. Partial Differential Equations}, 57(5):Paper No. 127,
  28, 2018.

\bibitem{CKY2018}
K.~Craig, I.~Kim, and Y.~Yao.
\newblock Congested aggregation via {N}ewtonian interaction.
\newblock {\em Arch. Ration. Mech. Anal.}, 227(1):1--67, 2018.

\bibitem{noemi2023}
N.~David.
\newblock Phenotypic heterogeneity in a model of tumour growth: existence of
  solutions and incompressible limit.
\newblock {\em Comm. Partial Differential Equations}, 48(4):678--710, 2023.

\bibitem{DDB2022}
N.~David, T.~Debiec, and B.~Perthame.
\newblock Convergence rate for the incompressible limit of nonlinear
  diffusion-advection equations.
\newblock {\em Ann. Inst. H. Poincar\'{e} C Anal. Non Lin\'{e}aire},
  40(3):511--529, 2023.

\bibitem{NAF2024}
N.~David, A.~R. Mészáros, and F.~Santambrogio.
\newblock Improved convergence rates for the {H}ele-{S}haw limit in the
  presence of confining potentials.
\newblock {\em arXiv preprint arXiv:2405.07227}, 2024.

\bibitem{DAVID2021}
N.~David and B.~Perthame.
\newblock Free boundary limit of a tumor growth model with nutrient.
\newblock {\em J. Math. Pures Appl. (9)}, 155:62--82, 2021.

\bibitem{DS2021}
N.~David and M.~Schmidtchen.
\newblock On the incompressible limit for a tumour growth model incorporating
  convective effects.
\newblock {\em Comm. Pure Appl. Math.}, 77(5):2613--2650, 2024.

\bibitem{DPSV2021}
T.~Debiec, B.~Perthame, M.~Schmidtchen, and N.~Vauchelet.
\newblock Incompressible limit for a two-species model with coupling through
  {B}rinkman's law in any dimension.
\newblock {\em J. Math. Pures Appl. (9)}, 145:204--239, 2021.

\bibitem{DS2020}
T.~Debiec and M.~Schmidtchen.
\newblock Incompressible limit for a two-species tumour model with coupling
  through {B}rinkman's law in one dimension.
\newblock {\em Acta Appl. Math.}, 169:593--611, 2020.

\bibitem{DQZ2020}
Y.~Du, F.~Quir\'{o}s, and M.~Zhou.
\newblock Logarithmic corrections in {F}isher-{KPP} type porous medium
  equations.
\newblock {\em J. Math. Pures Appl. (9)}, 136:415--455, 2020.

\bibitem{fhlz2024}
Y.~Feng, Q.~He, J.-G. Liu, and Z.~Zhou.
\newblock Rigorous derivation of a {H}ele-{S}haw type model and its
  non-symmetric traveling wave solution.
\newblock {\em arXiv preprint arXiv:2404.16353}, 2024.

\bibitem{gn1983}
D.~Gilbarg and N.~S. Trudinger.
\newblock {\em Elliptic Partial Differential Equations of Second Order}.
\newblock Springer, Berlin, 1983.

\bibitem{GKM2022}
N.~Guillen, I.~Kim, and A.~Mellet.
\newblock A {H}ele-{S}haw limit without monotonicity.
\newblock {\em Arch. Ration. Mech. Anal.}, 243(2):829--868, 2022.

\bibitem{HLP2023}
Q.~He, H.-L. Li, and B.~Perthame.
\newblock Incompressible limit of porous media equation with chemotaxis and
  growth.
\newblock {\em arXiv preprint arXiv:2312.16869}, 2023.

\bibitem{HLP2022}
Q.~He, H.-L. Li, and B.~Perthame.
\newblock Incompressible limits of the {P}atlak-{K}eller-{S}egel model and its
  stationary state.
\newblock {\em Acta Appl. Math.}, 188:Paper No. 11, 53, 2023.

\bibitem{Hillen2009}
T.~Hillen and K.~J. Painter.
\newblock A user's guide to {PDE} models for chemotaxis.
\newblock {\em J. Math. Biol.}, 58(1-2):183--217, 2009.

\bibitem{IN2021}
N.~Igbida.
\newblock ${L}^{1}$-theory for incompressible limit of reaction-diffusion
  porous medium flow with linear drift.
\newblock {\em arXiv preprint arXiv:2112.10411}, 2021.

\bibitem{keller1970initiation}
E.~F. Keller and L.~A. Segel.
\newblock Initiation of slime mold aggregation viewed as an instability.
\newblock {\em J. Theoret. Biol.}, 26(3):399--415, 1970.

\bibitem{KM2023}
I.~Kim and A.~Mellet.
\newblock Incompressible limit of a porous media equation with bistable and
  monostable reaction term.
\newblock {\em SIAM J. Math. Anal.}, 55(5):5318--5344, 2023.

\bibitem{KPW2019}
I.~Kim, N.~Po\v{z}\'{a}r, and B.~Woodhouse.
\newblock Singular limit of the porous medium equation with a drift.
\newblock {\em Adv. Math.}, 349:682--732, 2019.

\bibitem{KT2018}
I.~Kim and O.~Turanova.
\newblock Uniform convergence for the incompressible limit of a tumor growth
  model.
\newblock {\em Ann. Inst. H. Poincar\'{e} C Anal. Non Lin\'{e}aire},
  35(5):1321--1354, 2018.

\bibitem{IK2003}
I.~C. Kim.
\newblock Uniqueness and existence results on the {H}ele-{S}haw and the
  {S}tefan problems.
\newblock {\em Arch. Ration. Mech. Anal.}, 168(4):299--328, 2003.

\bibitem{kowalczyk2005preventing}
R.~Kowalczyk.
\newblock Preventing blow-up in a chemotaxis model.
\newblock {\em J. Math. Anal. Appl.}, 305(2):566--588, 2005.

\bibitem{LX2021}
J.-G. Liu and X.~Xu.
\newblock Existence and incompressible limit of a tissue growth model with
  autophagy.
\newblock {\em SIAM J. Math. Anal.}, 53(5):5215--5242, 2021.

\bibitem{MRS2010}
B.~Maury, A.~Roudneff-Chupin, and F.~Santambrogio.
\newblock A macroscopic crowd motion model of gradient flow type.
\newblock {\em Math. Models Methods Appl. Sci.}, 20(10):1787--1821, 2010.

\bibitem{MPQ2017}
A.~Mellet, B.~Perthame, and F.~Quir\'os.
\newblock A {H}ele-{S}haw problem for tumor growth.
\newblock {\em J. Funct. Anal.}, 273(10):3061--3093, 2017.

\bibitem{Murray2003}
J.~D. Murray.
\newblock {\em Mathematical biology. {II}}, volume~18 of {\em Interdisciplinary
  Applied Mathematics}.
\newblock Springer-Verlag, New York, third edition, 2003.
\newblock Spatial models and biomedical applications.

\bibitem{Painter2018}
K.~J. Painter.
\newblock Mathematical models for chemotaxis and their applications in
  self-organisation phenomena.
\newblock {\em J. Theoret. Biol.}, 481:162--182, 2019.

\bibitem{Hillen2002}
K.~J. Painter and T.~Hillen.
\newblock Volume-filling and quorum-sensing in models for chemosensitive
  movement.
\newblock {\em Can. Appl. Math. Q.}, 10(4):501--543, 2002.

\bibitem{patlak1953random}
C.~S. Patlak.
\newblock Random walk with persistence and external bias.
\newblock {\em Bull. Math. Biophys.}, 15:311--338, 1953.

\bibitem{bp2002}
B.~Perthame.
\newblock {\em Kinetic formulation of conservation laws}, volume~21.
\newblock Oxford University Press, 2002.

\bibitem{bp2009}
B.~Perthame and A.-L. Dalibard.
\newblock Existence of solutions of the hyperbolic {K}eller-{S}egel model.
\newblock {\em Trans. Amer. Math. Soc.}, 361(5):2319--2335, 2009.

\bibitem{PQTV2014}
B.~Perthame, F.~Quir\'{o}s, M.~Tang, and N.~Vauchelet.
\newblock Derivation of a {H}ele-{S}haw type system from a cell model with
  active motion.
\newblock {\em Interfaces Free Bound.}, 16(4):489--508, 2014.

\bibitem{PERTHAME2014}
B.~Perthame, F.~Quir\'{o}s, and J.~L. V\'{a}zquez.
\newblock The {H}ele-{S}haw asymptotics for mechanical models of tumor growth.
\newblock {\em Arch. Ration. Mech. Anal.}, 212(1):93--127, 2014.

\bibitem{PV2015}
B.~Perthame and N.~Vauchelet.
\newblock Incompressible limit of a mechanical model of tumour growth with
  viscosity.
\newblock {\em Philos. Trans. Roy. Soc. A}, 373(2050):20140283, 16, 2015.

\bibitem{bpmz2024}
B.~Perthame and M.~Zhang.
\newblock Patterns in the {K}eller-{S}egel system with density cut-off.
\newblock {\em Preprint}, 2024.

\bibitem{simon1986}
J.~Simon.
\newblock Compact sets in the space {$L^p(0,T;B)$}.
\newblock {\em Ann. Mat. Pura Appl. (4)}, 146:65--96, 1987.

\bibitem{Maini2008}
M.~J. Tindall, P.~K. Maini, S.~L. Porter, and J.~P. Armitage.
\newblock Overview of mathematical approaches used to model bacterial
  chemotaxis. {II}. {B}acterial populations.
\newblock {\em Bull. Math. Biol.}, 70(6):1570--1607, 2008.

\bibitem{ZT2014}
J.~Tong and Y.~P. Zhang.
\newblock Convergence of free boundaries in the incompressible limit of tumor
  growth models.
\newblock {\em arXiv preprint arXiv:2403.05804}, 2024.

\bibitem{wang2007classical}
Z.~Wang and T.~Hillen.
\newblock Classical solutions and pattern formation for a volume filling
  chemotaxis model.
\newblock {\em Chaos}, 17(3):037108, 13, 2007.

\end{thebibliography}

\end{document}